\documentclass[twoside]{article} 
\usepackage[accepted]{aistats2017}

\usepackage{caption}
\usepackage{subcaption}
\RequirePackage{graphicx}
\usepackage{subcaption}
\usepackage[utf8]{inputenc} % allow utf-8 input
\usepackage[T1]{fontenc}    % use 8-bit T1 fonts

\usepackage{bm,color}
\usepackage{graphicx}
\usepackage{url}            % simple URL typesetting
\usepackage{booktabs}       % professional-quality tables
\usepackage{amsfonts}       % blackboard math symbols
\usepackage{nicefrac}       % compact symbols for 1/2, etc.
\usepackage{microtype,arydshln}      % microtypography
\usepackage{graphicx} % more modern
\RequirePackage[colorlinks,citecolor=blue,urlcolor=blue]{hyperref}
\usepackage{amsmath,amssymb,rotating,multirow, amsthm}
\usepackage{lscape, pbox, float, hyperref, amsmath, bbm, amssymb}
\usepackage{subcaption}
\usepackage{graphicx, listings, color}
\usepackage{hyperref}       % hyperlinks
\usepackage{url}            % simple URL typesetting
\usepackage{booktabs}       % professional-quality tables
\usepackage{amsfonts}       % blackboard math symbols
\usepackage{nicefrac}       % compact symbols for 1/2, etc.
\usepackage{microtype,arydshln}      % microtypography
\usepackage{graphicx} % more modern
\usepackage{multirow}
\usepackage{mathrsfs}
\usepackage{amssymb,color,bm}
\usepackage{lscape, pbox, float, hyperref, amsmath, bbm}
\usepackage{graphicx, listings, color}
\usepackage[english]{babel}
\usepackage[normalem]{ulem}  % for striking out text with \sout{...}
\usepackage[left]{lineno}
\usepackage{mathrsfs}
\usepackage{amsmath,amssymb,rotating,multirow}

\newtheorem{rem}{Remark}
\newtheorem{theorem}{Theorem}

\newcommand{\bX}{\bm{X}}
\newcommand{\bY}{\boldsymbol{Y}}
\newcommand{\lam}{\boldsymbol{\Lambda}}
\newcommand{\by}{\boldsymbol{y}}
\newcommand{\bz}{\boldsymbol{z}}

\begin{document}

\runningtitle{Performance Bounds for Graphical Record Linkage}


\twocolumn[

\aistatstitle{Performance Bounds for Graphical Record Linkage}

\aistatsauthor{ Rebecca C. Steorts \And Matt Barnes \And  Willie Neiswanger }

\aistatsaddress{ Departments of Statistical Science \\
and Computer Science\\
Duke University \\
\texttt{\small beka@stat.duke.edu}
 \And  
The Robotics Institute \\
Carnegie Mellon University \\
\texttt{\small mbarnes1@cs.cmu.edu}
 \And 
 Machine Learning Department \\
Carnegie Mellon University \\
\texttt{\small willie@cs.cmu.edu}} ]

%\aistatsauthor{Rebecca C. Steorts \And Matt Barnes \And Willie Neiswanger}
%
%\aistatsaddress{Departments of Statistical \\Science and Computer Science\\ Duke University\\ beka@duke.stat.edu \And Department of Computer Science\\ Carnegie Mellon University\\ email \And Department of Computer Science\\ Carnegie Mellon University\\ email  } ]

\begin{abstract}
Record linkage involves merging
records in large, noisy databases to remove duplicate entities. 
It has become
an important area because of its widespread occurrence in bibliometrics, public health, official statistics production, political science, and beyond. Traditional linkage methods 
directly linking records to one another are computationally infeasible
as the number of records grows.
As a result, it is increasingly common for researchers to treat record linkage as a
clustering task, in which each latent entity is associated with one or more
noisy database records. We critically assess performance bounds
 using the  Kullback-Leibler (KL) divergence under a Bayesian record linkage framework, making connections to Kolchin partition models. We provide an upper bound using the KL divergence and a lower bound on the minimum probability of misclassifying a latent entity. We give insights for when our bounds hold using simulated data and provide practical user guidance. 
 
\end{abstract}

\section{Introduction}
\label{sec:intro}
Record linkage (de-deduplication or entity resolution) involves identifying duplicate
records in large, noisy databases~\cite{christen_2011}.
Traditional linkage methods that
directly link records to one another become computationally infeasible
as the number of records grows~\cite{christen_2011,winkler_2006}, and thus, it is
increasingly common for researchers to treat linkage as a
clustering task, in which latent entities are associated with one or more
noisy database records, and the inferential goal is to identify the
latent entity underlying each observed database
record~\cite{steorts15entity, steorts14smered, steorts??bayesian}. 
%Thus, record linkage involves clustering a set of database records such that the records in
%a single cluster correspond to a single latent entity. 
Although there are
many probabilistic, generative models for clustering --- of which
several have been used for record linkage --- the
theoretical properties, such as performance bounds, have such not
been critically assessed.

The work of \cite{steorts15entity, steorts14smered, steorts??bayesian} attempted to deconstruct  distorted data by latent variable mixture models. The authors achieved this by clustering similar records to a hypothesized latent entity for each observed record, where their \emph{linkage structure} kept track of which latent entity belongs to the same observed records. This is modeled through a latent variable mixture model with a distortion process on the data {(sections \ref{sec:categorical} and \ref{sec:string})}. Thus, the main goal is to be able to take distorted data and uncover the underlying structure in the presence of noise. This is similar to signal processing, where a signal is received in the presence of some noise and often the goal is to understand if the underlying true (latent) signal can be recovered. 
%We make this connection as performance bounds of record linkage, to our knowledge, have not been investigated. 
We develop performance bounds under the framework proposed by \cite{steorts15entity, steorts14smered, steorts??bayesian}.
%since their ``distortion process" has a clear connection to the signal processing literature. 
%since the idea of the distortion process and latent variable modeling has a clear connection. 

%To investigate such performance bounds, 
% We use the  Kullback-Leibler (KL) divergence to (i) provide an upper bound on the KL divergence and 
We provide an upper bound on the Kullback-Leibler (KL) divergence between models 
with different linkage structures and use it to provide a lower bound  on the minimum probability of misclassifying a latent entity. More precisely, under the categorical model of \cite{steorts14smered,steorts??bayesian} and string model of \cite{steorts15entity}, we find the minimum probability of getting a latent entity incorrect. 
%We find such a bound by deriving tail probabilities of the distribution of errors regarding the linkage structure, which provides linkages between similar records and a latent entity. We are able to show that when the amount of noise or distortion grows, then the latent entities become less similar. Next, under the model of \cite{steorts_2014_eb}, we also find 
%the minimum probability of getting a latent entity incorrect. In this situation, we find that the minimum probability of getting a latent incorrect is governed by two driving forces, namely the distance comparator of the noisy string variables as well as a constant that normalizes the distance comparator into a proper density. 
We make connections to Kolchin partition (KP) models \cite{pitman}, along with extending our overall KL bounds in general. Finally, we explore how our bounds perform in practice and describe their user practicality.

%forces due to the noisy string data. 
%[[need to find a way to state this in words]]. [[Then talk about challenges of these approaches or something]]. 

% In short,  we derive bounds based upon Kullback Leibler  divergence (KL), illustrating our bounds hold on both real and synthetic data. Next, we then extend to the categorical and string model of \citep{steorts_2014_eb}, where we derive similar bounds as just outlined. 

\subsection{Prior work}
Bayesian methods and latent variable modeling have become recently popularized in record linkage models. A major advantage of Bayesian methods is their natural handling of uncertainty quantification for the resulting estimates. The first notion of understanding a distortion process for record linkage is the hit-miss-model,  which uses a binary distortion process on the data \cite{copas_1990}.  Within the Bayesian paradigm, most work has focused on specialized approaches related to linking two files \cite{gutman_2013, liseo_2011}.  These contributions, while valuable, do not easily generalize to more than two files or to de-duplication within a single file. For a review of recent development in Bayesian methods, see \cite{liseo_2013}. 

The work of \cite{steorts14smered,steorts??bayesian} recently introduced a Bayesian model that simultaneously handled record linkage and de-duplication for categorical data. Their approach allowed for natural uncertainty quantification during analysis and post-processing.  
%They developed a framework for reporting a point estimate of the linkage structure. 
%This generality as well as their scalable algorithm makes exploration of performance bounds for record linkage feasible, and paves possible future work for general types of Bayesian clustering models.  
Finally, \cite{sadinle_2014} recently extended the  work of \cite{steorts??bayesian} to both categorical and string valued data using a coreference matrix or a partitioning approach. In the later paper, it was shown that the coreference matrix is a special case of the linkage structure, thus, we work with the linkage structure. 
Another advantage of \cite{steorts??bayesian} and similar approaches is that their linkage structure is amenable to an efficient MCMC inference algorithm. 
{These models have become practically relevant as they have been shown to perform well on a variety of applications, including official statistics and medical data. In addition, extensions have been made to more general framework of models \cite{sadinle_2014, liseo_2011, zanella2016microclustering}, which is incorporated into our framework in section \ref{sec:gibbs}.}

%these methods are more scalable when the data is categorical since a Gibbs sampler
%does not explore the parameter space as efficiently as a hybrid MCMC approach.  

Given the noted distortion process, deriving performance bounds seems  natural to recover the underlying structure. For example, much work has been done in information theory for subset selection in graphical model selection, signal de-noising, compressive sensing, and others. 
%\cite{wainwright2007}. 
In compressed sensing, one question recently addressed in \cite{donoho2006}, was directly measuring the part of the data from sounds and images that \emph{will not} be thrown away. We make a connection here, as in record linkage we wish to take noisy, distorted data and recover this under the KL divergence. Divergence functions by \cite{shannon1948, kullback_1951} are useful in many applications including recent statistical applications of clustering, as done in \cite{banerjee2005} for hard clustering to obtain optimal quantization by minimizing the Bregman divergence (motivated by rate distortion theory). \\

The rest of this paper proceeds as follows. 
Two recent record linkage models are given in section~\ref{sec:background}; Section~\ref{sec:categorical} and section~\ref{sec:string} review these  models. 
 Section~\ref{sec:properties} derives the respective performance bounds, while section \ref{sec:gibbs} extends our general result to a wider class of models. Section \ref{sec:experiments} shows performance of the bounds in practice, discusses our findings and user practicality. 
  Section \ref{sec:discussion} discusses future work.
 
\vspace{-2mm}
\section{Bayesian Record Linkage}
\label{sec:background}
We assume two Bayesian record linkage models, one dealing with categorical data and the other dealing with both categorical and noisy string data, such as names, addresses, etc. 
The first is that of \cite{steorts14smered,steorts??bayesian}, and the second is that of \cite{steorts15entity}.
%%%% WILLIE'S NOTE: IM NOT SURE IF WE NEED THE NEXT PARAGRAPH. jCOMMENTING OUT FOR NOW!
%\textcolor{red}{We consider these two models due to their relationships with the signal processing literature, namely the use of the aforementioned models has a distortion process embedded, which to our knowledge does not appear in the rest of the Bayesian literature. That is, all other models assume there is no distortion process of the data. }

%%%% OLD PARAGRAPH
%%Both papers have been shown to work on a wide range of data, such as medical studies, official statistics, and simulated data sets. 
%%In this paper, our goal is understanding the underlying performance bounds of both the aforementioned papers, which will hope to further more theoretical developments in record linkage and related areas. 

\subsection{Categorical Bayesian Record Linkage}
\label{sec:categorical}

We review common notation to both models.\footnote{For a toy example of the record linkage process, see the Supplementary Material.}
{Let $\bX=(X_1,\ldots,X_n)$} represent the data, with $k$ databases, indexed by~$i$.  The $i$th list has $n_i$ observed records, indexed by~$j$.  Each record corresponds to one of $N$ latent entities, indexed by $j'.$ Assume
 $N=\sum_{i=1}^k n_i$ without loss of generality.
% since there can be at most $\sum_{i=1}^k n_i$ distinct latent entities to which any record refers.
Each record or latent entity has values on $p$~fields, indexed by~$\ell$, and are assumed  be categorical and the same across all records and entities \cite{steorts14smered,steorts??bayesian}.
$M_\ell$ denotes the number of possible categorical values for the $\ell$th field.

In both models, $X_{ij\ell}$ denotes the
observed value of the $\ell$th field for the $j$th record in the $i$th list,
and $Y_{j'\ell}$ denotes the true value of the $\ell$th field for the $j'$th latent
entity. Then $\Lambda_{ij}$ denotes the latent entity to which the
$j$th record in the $i$th list corresponds, i.e., $X_{ij\ell}$ and $Y_{j'\ell}$
represent the same entity if and only if $\Lambda_{ij}=j'$.
Then $\bm\Lambda$ denotes the $\Lambda_{ij}$ collectively.
Distortion is denoted by $z_{ij\ell}=I(X_{ij\ell}\ne Y_{\Lambda_{ij}\ell})$,
where $I(\cdot)$ denotes the indicator function.
As usual,  $I$ represents the indicator function (e.g., $I(x_{ij\ell}=m)$ is 1 when
the $\ell$th field in record $j$ in file $i$ has the value $m$),
and let $\delta_a$ denote the distribution of a point mass at $a$ (e.g., $\delta_{y_{\Lambda_{ij}\ell}}$). 
%$\text{Mulitnomial}(\bml)$ denotes the multinomial distribution with probabilities given by the components of the vector~$\bm\theta_\ell$, which has length~$M_\ell$.
The model of \cite{steorts14smered,steorts??bayesian} is: 
%\begin{align}
%\label{model:cat}
%X_{ij\ell}\mid\Lambda_{ij},Y_{\Lambda_{ij}\ell},z_{ij\ell},\bm\theta_\ell&\stackrel{\text{ind}}{\sim}
%\begin{cases}
%\delta_{Y_{\Lambda_{ij}\ell}}&\text{ if }z_{ij\ell}=0 \notag \\
%\text{Multinomial}(1,\bm{\theta}_\ell)&\text{ if }z_{ij\ell}=1
%\end{cases}\\
%z_{ij\ell}&\stackrel{\text{ind}}{\sim}\text{Bernoulli}(\beta_\ell) \notag \\
%Y_{j'\ell}\mid\bm{\theta}_{\ell}&\stackrel{\text{ind}}{\sim}\text{Multinomial}(1,\bm{\theta}_\ell) \notag \\
%\bm{\theta}_\ell&\stackrel{\text{ind}}{\sim}\text{Dirichlet}(\bm{\mu}_\ell) \notag \\
%\beta_\ell &\stackrel{\text{ind}}{\sim}\text{Beta}(a_\ell,b_\ell) \notag \\
%\pi(\lam) \propto 1,
%\end{align}
\begin{align}
\label{model:cat}
%\bm{x}_{ij\ell}\mid\Lambda_{ij},\bm{y}_{\Lambda_{ij}\ell},z_{ij\ell},\bm{\theta}_\ell&\stackrel{\text{ind}}{\sim}
%\begin{cases}
%\delta_{\bm{y}_{\Lambda_{ij}\ell}}&\text{ if }z_{ij\ell}=0\\
%\text{MN}(1,\bm{\theta}_\ell)&\text{ if }z_{ij\ell}=1
%\end{cases}\\
%z_{ij\ell}&\stackrel{\text{ind}}{\sim}\text{Bernoulli}(\beta_\ell)\\
%\bm{y}_{j'\ell}\mid\bm{\theta}_{\ell}&\stackrel{\text{ind}}{\sim}\text{MN}(1,\bm{\theta}_\ell)\\
%\bm{\theta}_\ell&\stackrel{\text{ind}}{\sim}\text{Dirichlet}(\bm{\mu}_\ell)\\
%\beta_\ell&\stackrel{\text{ind}}{\sim}\text{Beta}(a_\ell,b_\ell) \\
%\pi(\lam) &\propto 1,
X_{ij\ell}\mid\Lambda_{ij},Y_{\Lambda_{ij}\ell},z_{ij\ell},\bm\theta_\ell&\stackrel{\text{ind}}{\sim}
\begin{cases}
\delta_{Y_{\Lambda_{ij}\ell}}&\text{ if }z_{ij\ell}=0 \notag \\
\text{MN}(1,\bm{\theta}_\ell)&\text{ if }z_{ij\ell}=1
\end{cases}\\
z_{ij\ell}&\stackrel{\text{ind}}{\sim}\text{Bernoulli}(\beta_\ell) \notag \\
Y_{j'\ell}\mid\bm{\theta}_{\ell}&\stackrel{\text{ind}}{\sim}\text{MN}(1,\bm{\theta}_\ell) \notag \\
\bm{\theta}_\ell\stackrel{\text{ind}}{\sim}\text{Dirichlet}(\bm{\mu}_\ell) & \;\; \text{and} \;\;
\beta_\ell\stackrel{\text{ind}}{\sim}\text{Beta}(a_\ell,b_\ell) \notag  \\
\Lambda_{ij} &\stackrel{\text{ind}}{\sim}\text{Uniform}\left(1,\ldots,N\right),
\end{align}
where MN denotes the Multinomial distribution and
$a_\ell,b_\ell, \bm{\mu}_\ell$ are all known. Guidance for the hyper-parameters and a justification of the (discrete) uniform prior are given in \cite{steorts14smered,steorts15entity, steorts??bayesian}.
Model \ref{model:cat} assumes that different records are independent conditional on the deeper variables of the model.  Moreover, it assumes the same conditional independence of different fields for the same record.  
%Note that that the prior on the linkage structure $\bm\Lambda$ can equivalently be represented as
%$
%\Lambda_{ij}\stackrel{\text{iid}}{\sim}\text{DiscreteUniform}\left(1,\ldots,N\right).
%$
Finally, observe that record linkage and de-duplication are both simply a question of whether $\Lambda_{i_1,j_1}=\Lambda_{i_2,j_2}$, where $i_1\ne i_2$ for record linkage and $i_1=i_2$ for de-duplication. 

\subsection{Empirical Bayesian Record Linkage}
\label{sec:string}
The work of \cite{steorts15entity} 
assumes  fields $1,\ldots,p_s$ are string-valued, while fields $p_s+1,\ldots,p_s+p_c$ are categorical, where $p_s+p_c=p$ is the total number of fields. They assume an empirical Bayesian distribution on the latent parameter. For each $\ell\in\{1,\ldots,p_s+p_c\}$, let $S_\ell$ denote the set of \emph{all} values for the $\ell$th field
that occur anywhere in the data, i.e.,
$S_\ell=\{X_{ij\ell}:1\le i\le k, 1\le j\le n_i\},$
and let $\alpha_\ell(w)$ equal the empirical frequency of value~$w$ in field~$\ell.$
Let $G_\ell$ denote the empirical distribution of the data in the $\ell$th field from all records in all databases combined.  So, if a random variable~$W$ has distribution $G_\ell$, then for every $w\in S_\ell$,
$P(W=w)=\alpha_\ell(w)$.
Hence, let $G_\ell$  be the prior for each latent entity $Y_{j'\ell}$. 
%
%Due to such changes in the model, the distortion process also changes to the following: 
The distortion process changes such that
\begin{align*}
%\label{model:string}
P(X_{ij\ell}&=w\mid\Lambda_{ij},Y_{\Lambda_{ij}\ell},z_{ij\ell}) \\
&=\dfrac{\alpha_\ell(w)\,\exp[-c\,d(w,Y_{\Lambda_{ij}\ell})]}{\sum_{w\in S_\ell}\alpha_\ell(w)\,\exp[-c\,d(w,Y_{\Lambda_{ij}\ell})]},
\end{align*}
where $c > 0$ is a fixed normalizing constant corresponding to an arbitary distance metric $d(\cdot,\cdot)$.  Denote this distribution by $F_\ell(Y_{\Lambda_{ij}\ell})$.
The model becomes
\begin{align}
\label{model:string}
X_{ij\ell}\mid \Lambda_{ij},\,Y_{\Lambda_{ij}\ell},\,z_{ij\ell}\;&\stackrel{\text{ind}}{\sim}\begin{cases}\delta(Y_{\Lambda_{ij}\ell})&\text{ if }z_{ij\ell}=0\\F_\ell(Y_{\Lambda_{ij}\ell})&\text{ if }z_{ij\ell}=1, \ell\le p_s\\G_\ell&\text{ if }z_{ij\ell}=1, \ell>p_s\end{cases} \notag \\
%&\qquad\text{for each }i\in\{1,\ldots,k\},\; j\in\{1,\ldots,n_i\},\; \ell\in\{1,\ldots,p_s+p_c\},\\
%&\qquad\text{with everything independent},\\
Y_{j'\ell}\;&\stackrel{\text{ind}}{\sim}G_\ell \notag \\
z_{ij\ell}\mid\beta_{i\ell}\;&\stackrel{\text{ind}}{\sim}\text{Bernoulli}(\beta_{i\ell}) \notag\\
\beta_{i\ell}\;&\stackrel{\text{ind}}{\sim}\text{Beta}(a,b) \notag \\ 
\Lambda_{ij}\;&\stackrel{\text{ind}}{\sim}\text{Uniform}\left(1,\ldots,N\right),
\end{align}
%where $i\in\{1,\ldots,k\}$, $j\in\{1,\ldots,n_i\}$, $\ell\in\{1,\ldots,p_s+p_c\}$, and $j'\in\{1,\ldots,N\}$, and
where all distributions are also independent of each other; assume that $a,b, N$ are assumed known. This framework was shown to work well in applications and simulation studies, however, it was quite sensitive to the choice of the hyperparameters. This method  beat supervised methods, such as random forests when the amount of training data input into the supervised methods was $< 10\%$. 

Figure~\ref{fig:graphicalProcess} contains a graphical representation of models \ref{model:cat}-\ref{model:string}.

\begin{figure}[htbp]
\begin{center}
\includegraphics[width=0.35\textwidth]{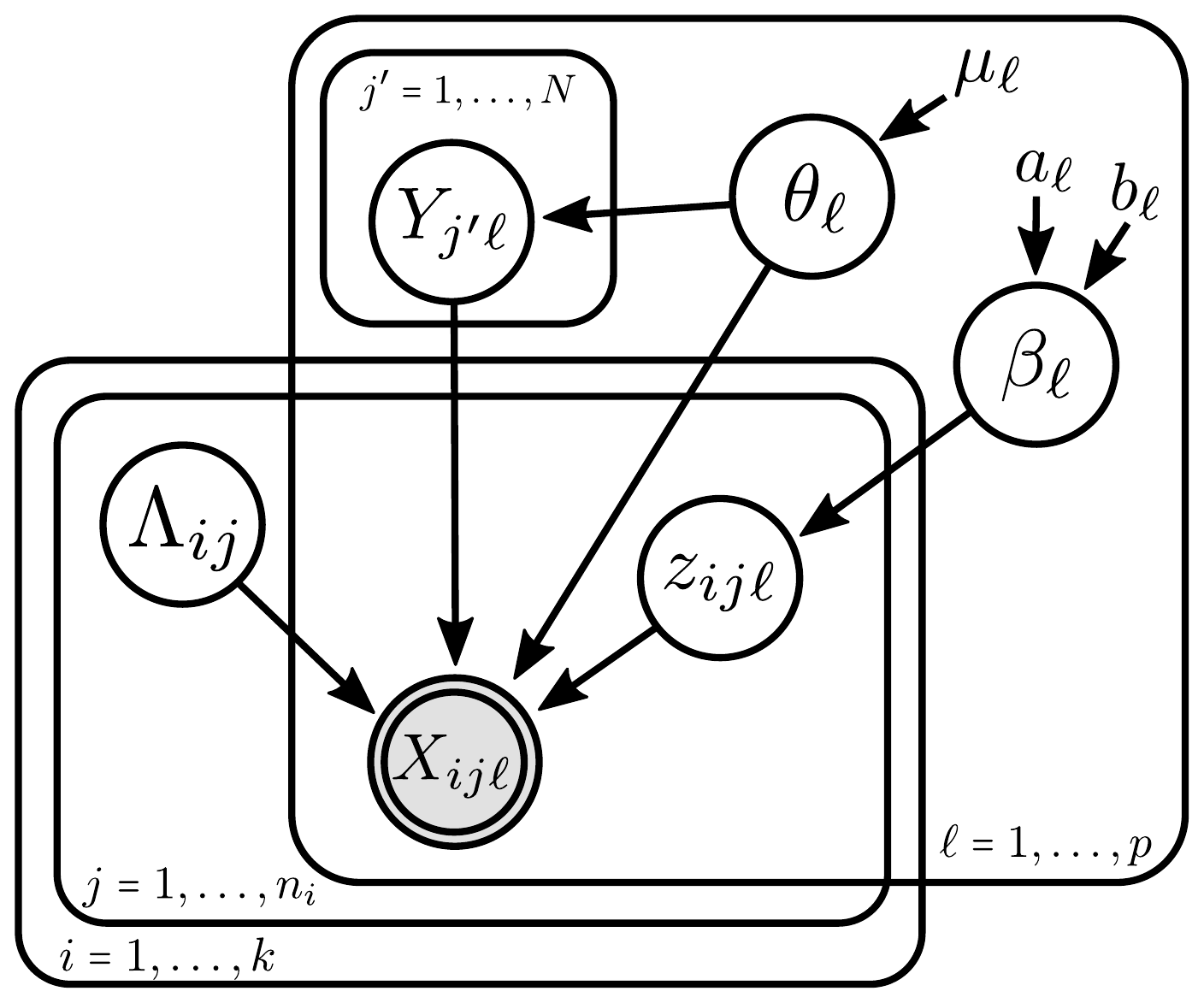}
\caption{Graphical representation of models \ref{model:cat}-\ref{model:string}.}
\label{fig:graphicalProcess}
\end{center}
\end{figure}

%\textcolor{red}{The Supplementary Material \ref{sec:graphic-pic} contains a graphical representation of equations \ref{model:cat} and \ref{model:string}.}
%

% illustrating that the method works very well for applications where very little training data is available. 

\section{Performance Bounds of Record Linkage}
\label{sec:properties}

Recall the connection to KL divergence in the sense that for any two distributions $P$ and $Q$, the maximum power for testing $P$ versus $Q$ is $\exp\{-n D_{\text{KL}}(P || Q)\}.$ Hence, a low value of $D_{KL}$ means that we need many samples to distinguish $P$ from $Q.$ A natural question is how  does changing $\bY$ (latent entity) or $\lam$ (linkage structure) change the distribution of $\bX$ (observed records)? We search for both meaningful upper and lower bounds, since an upper bound will say that $P$ and $Q$ are never more than so far apart, whereas a lower bound says how easy it is to tell $P$ and $Q$ apart. Moreover, we investigate how well can we recover $\bY$ (latent entity) and $\lam$ (linkage structure) from $\bX$ (data).

%We investigate how well can we recover $\bY$ (latent entity) and $\lam$ (linkage structure) from $\bX$ (data). For any two distributions $P$ and $Q$, it is well known that the maximum power for testing $P$ versus $Q$  \cite{kullback_1951}
% is $O(\exp\{- N \; D_{\text{KL}}(P || Q)\}).$ Hence, low $D_{\text{KL}}$ means it is intrinsically hard to tell $P$ from $Q.$ 
 
Assuming the conditions of \cite{steorts14smered, steorts15entity}, 
let 
 $\mathcal{P} = \left\{f(X\mid \bY, \Lambda_{ij}, \bm{\theta}, \bm{\beta}): 
 \forall \Lambda_{ij} \in \{ 1, \ldots, N \}.\right\}$
% Given $P, Q \in \mathcal{P}$, 
We know that $X_1,X_2,\ldots,X_N$ are all independent given $(\bY,\lam, \bm{\theta}, \bm{\beta})$ under both $P, Q \in \mathcal{P}.$ This implies that 
$D_{X_1, X_2, \ldots, X_N} (P \| Q) = \sum_i D_{X_i}(P \| Q).$
  We first provide a theorem under the model of \cite{steorts14smered}, which assumes categorical data and a hierarchical model. In Theorem~\ref{theorem:cat}, we find the minimum probability of getting a latent entity wrong. Moreover, we are able to say that with growing distortion of the data, there is no difference between two latent entities and the bound becomes infinite and non-informative in this case. 
Next, under the model of \cite{steorts15entity} we provide a general theorem, which assumes both categorical and noisy text data. This theorem provides an upper bound on the KL divergence of arbitrary distributions $P$ and $Q$.

\subsection{Kullback-Leibler Divergence under Categorical Data} \label{sec:cat}
We use Fano's inequality \cite{prelov2008} to bound the probability of misclassification, as a function of the KL divergence between $P$ and $Q$, as defined in the previous section.
Assume that $\lam$ and $\hat{\lam}$ are two distinct linkage structures that correspond to the same latent entity $(\by).$ 
Let $r+1$ be the cardinality of $\mathcal{P}$, i.e.\ $r+1 = N$.

\begin{theorem}
\label{theorem:cat}
This result finds an upper bound on the KL divergence and a
 lower bound for the  probability that model \ref{model:cat} gets the linkage structure incorrect. 
Let
$\gamma = \max_{\Lambda_{ij} \neq \Lambda'_{ij}}
2\sum_{ij\ell} I(Y_{\Lambda_{ij}\ell} \neq Y_{\Lambda'_{ij}\ell}) (1-\beta_{\ell}) \ln \left \{
\dfrac{1}{
\min_m \theta_{\ell m} \beta_{\ell}} \right\}.$
 \begin{enumerate}
\item[i)] The KL divergence is bounded above by $\gamma.$ That is,
$D_X(P || Q) \leq \gamma \enskip \forall P, Q \in \mathcal{P}$.
\item[ii)] The minimum probability of getting a latent entity wrong is
$Pr( {\Lambda}_{ij} \ne \Lambda^\prime_{ij}) \geq 1 - \dfrac{ \gamma + \ln 2}{\ln r}, \enskip \forall i,j$
\end{enumerate}
\end{theorem}
That is, as the latent entities become more distinct, $\gamma$ increases. On the other hand, as the latent entities become more similar, $\gamma \rightarrow 0.$
%
%\begin{rem}
%Consider Theorem \ref{theorem:cat} (i). Suppose $\beta_{\ell} \rightarrow 1.$ Then the bound reduces to $D_{\bX}=0.$ This implies that when there is growing distortion, there is no difference between two latent entities. Suppose $\beta_{\ell} \rightarrow 0.$ Then the upper bound is infinite, which is not informative here. 
%\end{rem}
%%
{

\begin{rem}
%\textbf{Remark}:
Consider Theorem \ref{theorem:cat} (i). Suppose $\beta_{\ell} \rightarrow 1.$ Then $D_{\bX} \geq 0.$ If instead $\beta_{\ell} \rightarrow 0,$ then $D_{\bX} \geq 1.$ The lower bound is only informative when $\beta_{\ell} \rightarrow 0.$ We have more information when the latent entities are separated.
\end{rem}
\begin{proof}
To show this, we simply apply Pinsker's  inequality, where
for all $P,Q \in \mathcal{P}$:
$
D(P \| Q) \geq 2||P-Q||^2_1 
\implies \notag \\
D(P \| Q) \geq I(Y_{\Lambda_{ij}\ell} \neq Y_{\Lambda^\prime_{ij}\ell}) (1-\beta_{\ell})^2 
\nonumber \implies  \notag \\
 D_{\bX}(P \| Q)  \geq \sum_{ij\ell} I(Y_{\Lambda_{ij}\ell} \neq Y_{\Lambda^\prime_{ij}\ell})  (1-\beta_{\ell})^2.
$
\end{proof}
}
\begin{proof}
We assume the model of \cite{steorts14smered,steorts??bayesian}, which assumes that data is categorical.  We assume model \ref{model:cat} holds in section \ref{sec:categorical}. We first prove (i). Consider $f(X\mid \bY, \lam, \bm{\theta}, \bm{\beta}). $ Then
\begin{align}
\label{cat:like}
&Pr(X_{ij\ell} =m \mid \bY, \lam, \bm{\theta}, \bm{\beta})  \\
% &= 
% Pr(X_{ij\ell} =m, Z_{ij\ell} = 1 \mid \bY, \lam, \bm{\theta}, \bm{\beta}) \notag\\
%& \qquad+ Pr(X_{ij\ell} =m, Z_{ij\ell} = 0 \mid \bY, \lam, \bm{\theta}, \bm{\beta}) \notag\\
%&= Pr(X_{ij\ell} =m \mid \bY, \lam, \bm{\theta}, \bm{\beta}, Z_{ij\ell} = 1)  \notag\\
%&\times 
%Pr( Z_{ij\ell} = 1 \mid  \bY, \lam, \bm{\theta}, \bm{\beta})\notag \\
%&\qquad + Pr(X_{ij\ell} =m \mid \bY, \lam, \bm{\theta}, \bm{\beta}, Z_{ij\ell} = 0) \notag\\
%& \times Pr( Z_{ij\ell} = 0 \mid  \bY, \lam, \bm{\theta}, \bm{\beta})\notag \\
&= 1(Y_{\Lambda_{ij}\ell} = m) (1-\beta_{\ell}) 
+ \theta_{\ell m} \beta_{\ell}.\notag 
\end{align}
It follows from equation~\ref{cat:like} that 
\begin{align*}
D_{X_{ij\ell}}(P\| Q)
&= \sum_{m=1}^{M_\ell}  
I(Y_{\Lambda_{ij}\ell} = m) (1-\beta_{\ell}) +  \theta_{\ell m} \beta_{\ell}\} \\
& \times \log\left[
\frac{I(Y_{\Lambda_{ij}\ell} = m) (1-\beta_{\ell}) +  \theta_{\ell m} \beta_{\ell}}
{
I(Y_{\Lambda^\prime_{ij}\ell} = m) (1-\beta_{\ell}) +  \theta_{\ell m} \beta_{\ell}
}
\right]. 
\end{align*}
%Equation \ref{eqn:cat:div} implies that 
It directly follows that 
\begin{align*}
D_{\bX} (P\| Q) 
&=
\sum_{ij\ell m} \left\{
I(Y_{\Lambda_{ij}\ell} = m) (1-\beta_{\ell}) +  \theta_{\ell m} \beta_{\ell}\} \right. \\
&\times \left. \log\left[
\frac{I(Y_{\Lambda_{ij}\ell} = m) (1-\beta_{\ell}) +  \theta_{\ell m} \beta_{\ell}}
{
I(Y_{\Lambda^\prime_{ij}\ell} = m) (1-\beta_{\ell}) +  \theta_{\ell m} \beta_{\ell}
}
\right]
\right\}.
\end{align*}
\textcolor{black}{If $Y_{\Lambda_{ij}\ell} \neq Y_{\Lambda^\prime_{ij}\ell}$}, then
\begin{align}
%\label{eqn:pq}
\| P-Q\|_1 
&= \sum_m \left | I(Y_{\Lambda_{ij}\ell} = m) (1-\beta_{\ell})  
+  \theta_{\ell m} \beta_{\ell} \right. \notag \\
& \qquad \left. -I(Y_{\Lambda^\prime_{ij}\ell} = m) (1-\beta_{\ell}) -  \theta_{\ell m} \beta_{\ell}
\right |  \notag \\
%&= \sum_m \left | I(Y_{\Lambda_{ij}\ell} = m) (1-\beta_{\ell}) \right. \notag \\
%& \qquad \left. -I(Y_{\Lambda^\prime_{ij}\ell} = m) (1-\beta_{\ell})   
%\right | \notag \\
&= 2(1-\beta_\ell).
\label{eqn:pq}
\end{align}
Equation \ref{eqn:pq} holds 
since $P(m) = Q(m)$ unless $m = Y_{\Lambda_{ij}\ell}$ or $m = Y_{\Lambda^\prime_{ij}\ell}.$
\textcolor{black}{If $Y_{\Lambda_{ij}\ell} = Y_{\Lambda^\prime_{ij}\ell}$, then $P=Q$ and $\|P-Q\|_1 = 0$.}
The reverse Pinsker inequality of \cite{berend_2014} relates the KL
divergence to the $L_1$ norm in the following way:
$D(P \| Q) \leq \| P-Q\|_1 \ln \{ (\min Q)^{-1}\}.$ Using this, we find that \textcolor{black}{(if $Y_{\Lambda_{ij}\ell} \neq Y_{\Lambda^\prime_{ij}\ell}$)}, then
\begin{align*}
D(P \| Q)  &\leq 2(1-\beta_{\ell}) \\
\qquad &\times \ln \left \{
\dfrac{1}{
\min_m I(Y_{\Lambda^\prime_{ij}\ell} = m) (1-\beta_{\ell}) +  \theta_{\ell m} \beta_{\ell}
}
\right\}  \\
&\leq 
2(1-\beta_{\ell}) \ln \left \{
\dfrac{1}{
\min_m \theta_{\ell m} \beta_{\ell}
}
\right\}.
\end{align*}
Hence, 
\begin{align*}
\max_{P,Q \in \mathcal{P}}D_{\bX}(P \| Q) &\leq
\max_{\Lambda_{ij} \neq \Lambda^\prime_{ij}} 2\sum_{ij\ell} I(Y_{\Lambda_{ij}\ell} \neq Y_{\Lambda^\prime_{ij}\ell})(1-\beta_{\ell}) \\
& \qquad \times \ln \left \{\dfrac{1}{\min_m \theta_{\ell m} \beta_{\ell}} \right\}
:=\gamma.
\end{align*}
This proves (i). 
We now prove (ii). Using Fano's inequality \cite{prelov2008}, the minimum probability of getting a latent entity wrong is
$Pr( \Lambda_{ij} \ne \Lambda^\prime_{ij}) \geq 1 - \frac{ \gamma + \ln 2}{\ln r},$ where $r+1$ is the cardinality of $\mathcal{P}$, i.e.\ $r+1=N$. As the latent entities become more distinct, $\gamma$ increases. On the other hand, as the latent entities become more similar, $\gamma \rightarrow 0.$
\end{proof}

\subsection{KL Divergence Bounds for String and Categorical Data} \label{sec:str}
We now consider $P$ and $Q$ under \cite{steorts15entity} for both categorical and noisy string data. Recall that  $\beta_\ell$ tunes the amount of distortion as defined in equation \ref{model:string}. 
Recall that $d(\cdot, \cdot)$ denotes any arbitrary distance metric between an observed string and a latent string as seen in equation \ref{model:string}, and $c > 0$ is a fixed normalizing constant corresponding to the distance metric $d.$

%\onecolumn
In Theorem \ref{theorem:string}, for any distinct linkage structures,  the minimum probability of getting a latent entity wrong is governed by a lower bound, which is growing at a rate $c \rightarrow \infty$ that is determined by the moment generating function of the distances between an observed string in data and a latent string. 
\begin{theorem}
\label{theorem:string}
Assume data $\bX,$ and distributions $P,Q \in \mathcal{P}$ defined in section \ref{sec:properties}.  Assume two distinct linkage structures, denoted by $Y_{\Lambda_{ij}\ell}, Y_{\Lambda^\prime_{ij}\ell}.$
\begin{enumerate}
\item [i)] There is an upper bound on the KL divergence between any $P,Q \in \mathcal{P}$ 
given by $\kappa,$ that is $D_X(P||Q) \leq \kappa.$
%\begin{align}
%&D_{\bX}(P || Q)\\
% &\geq
%\sum_{i,j,\ell} 2(1 - \beta_\ell) \\
%& + \sum_{i,j,\ell}
%I(Y_{\Lambda_{ij}\ell} \neq Y_{\Lambda^\prime_{ij}\ell})
%\left(
%1 - e^{-c d(Y_{\Lambda_{ij}\ell}, Y_{\Lambda^\prime_{ij}\ell})}
%\right) E[ e^{-c  d(m, Y_{\Lambda_{ij}\ell})} ],
%\end{align}
\item [ii)] $Pr(\Lambda_{ij} \neq \Lambda^\prime_{ij}) \geq 1- \dfrac{\kappa + \ln 2}{\ln r},$
where 
\begin{align*}
\kappa &= \max_{\Lambda_{ij} \neq \Lambda^\prime_{ij}}\bigg[
2 \sum_{\ell} (1-\beta_\ell) I(Y_{\Lambda_{ij}\ell} \neq Y_{\Lambda^\prime_{ij}\ell}) 
  +  \\
& \qquad \sum_{\ell m}  I(Y_{\Lambda_{ij}\ell} \neq Y_{\Lambda^\prime_{ij}\ell}) 
 \left(
1 - e^{-c d(Y_{\Lambda_{ij}\ell}, Y_{\Lambda^\prime_{ij}\ell})}
\right) \\
&\times E[ e^{-c  d(m, Y_{\Lambda_{ij}\ell})} ] \bigg]\ln\{ (\min Q)^{-1} \}
\end{align*}
and $r+1$ is the cardinality of $\mathcal{P}$.
% where the expectation is taken according to 
%random variable $M \sim \alpha_\ell.$ That is, $\sum_m  \alpha_\ell (m)
%e^{-c  d(m, m^\prime)} $ is the moment generating function of $d(M,m^\prime)$ (evaluated at c).
\end{enumerate}
\end{theorem}

The proof of Theorem \ref{theorem:string} can be found in the Supplementary Material.
%\ref{app:theorem:string}. 

\section{Kolchin Partition Models}
\label{sec:gibbs}
{
The models in sections \ref{sec:categorical} and \ref{sec:string} assume discrete uniform priors on the linkage structure.  We extend this to a more general class of models from Bayesian nonparametrics known as KP models \cite{pitman}. Special cases include the work of \cite{zanella2016microclustering, liseo_2011, sadinle_2014}. We provide notation, examples, and then provide a general theorem. }

The prior structure on $\lam$ can instead be viewed on the set of labelings. Specifically, let  $z$ denote the partition of the observed records determined by $\lam$, and $\mathcal{B}$  denote the set containing all the possible partitions of the $r$ observed records. Then a distribution on the sample labels $\lam$ induces a distribution on $\mathcal{B}.$ That is, matches and duplicates are completely specified given the knowledge of $z$, which is invariant with respect to the labelings of the partition blocks. 
%As mentioned in \cite{sadinle_2014}, we can focus on the partition distribution of the observed records, without linking the labels distribution to a sample design and to a population size $N$. Thus, we can consider the distribution of $\lam$ as a prior distribution for the latent structure and concentrate only on its probabilistic properties.  
%From a practical point of view both the interpretations for the role of $\lam$ (either as a consequence of the sampling design or a model for partitions) this may provide useful insights for a correct choice of its prior distribution and providing tight bounds. 

%Both the uniform prior on the linkage structure and the uniform prior on the coreference matrix can be viewed as a prior on the space of partitions, as shown in 
%\cite{steorts??bayesian}. 
%Third, we propose the Pitman-Yor process prior on on the space of partitions for $\lam.$  

\subsection{Special Cases of KP Models}
We give some special cases of KP models that extend the class of models that we consider.  
%We consider a uniform prior on the linkage structure as proposed in \cite{steorts14smered, steorts??bayesian, steorts15entity} and a special case of the linkage structure, known as the coreference matrix \cite{sadinle_2014}. 

%To analyze the properties of this prior assumption we start by  noting that
%$$P(\lam_{i_1 j_1}=\lam_{i_2 j_2})=\frac1N \quad \forall \,\, (i_1, j_1)\neq ( i_2, j_2);$$
%in other words,   the expected number of links between  pairs of records both inside and across the $k$ lists is                                                                                                                                                                                                               
%$$l_{0}=\frac{1}{N} \binom{r}{2}.$$
%This information can be useful in those cases where the population size $N$ is unknown but one has some information about the expected number of matchings and duplications summarized by $l_0$, since one could set  
%$N \approx \binom{r}{2}/l_0$. In addition, an upper bound is always known about $N$ in practice, hence, it can simply be approximated to the be the total number of records. 

\subsection{A Uniform Prior on the Label Space}

Let $z(\lam)$ denote the partition identified by $\lam$ and let $n(z)$ denote the number of distinct entities considering all the observed records, i.e. the number of blocks of the partition. One has $$N!/(N-n(z))!=(N)_{n(z)}$$ different labelings which identify the same partition $z(\lam)$. Then  
$$ P(z(\lam)=z)=\left(\frac{1}{N} \right)^r \frac{N!}{(N-n(z) )!} \quad \forall z \in \mathcal{B} .$$
Note also that $N^r=\sum_{n=0}^r (N)_{n} S(r,n)$ where $S(r,n)$ is   the Stirling number of second kind,
that is the number of possible partitions of the $r$ records into $n$ non empty sets, which implies
\begin{equation}
P(z(\lam)=z)=\frac{(N)_{n}}{N^r}\quad \forall z \in \mathcal{B} \label{gibbs1}
\end{equation}
where $n=n(z)$. Following \cite{pitman}, equation (\ref{gibbs1}) is a special case of a KP model. 
Moreover, the distribution of the number of distinct elements $n(z)$ is given by
$P(n(z)=n)= \dfrac{(N)_{n} S(r,n)}{ N^r}.$ (A similar prior was considered in \cite{liseo_2011}).
%One can easily find  mean and variance of this distribution. They are 
%$$E(n(z))=N \left(1-(1-1/N)^r \right)$$
%and 
%$$Var(n(z))=N\left((N-1)(1-2/N)^r-N(1-1/N)^{2r}+(1-1/N)^r\right),$$
%respectively. 
%Note that when $N$ is fixed and  the number of records $r$ tends to infinity,  the distribution of $n(z)$ concentrates on $N$, since $E(n(z))$ tends to $N$ and the variance to 0. Observe also that when the number of records $r$ is fixed and $N$ tends to infinity, the distribution of $n(z)$ concentrates on $r$, that is the prior probability to observe duplications or links converges to 0. 

\subsection{The Uniform Prior on the Partition Space}
\label{pyprior}
Assuming one database, \cite{sadinle_2014}  focused mainly on the partitions of the $r$ records induced by $\lam$ and proposed a flat prior on the partition space, that is a prior which assigns equal probability to each different partition of the $r$ observed records. Assume that
$$p(\lam)=\dfrac{1}{B_r \, (N)_{n(z)}} $$
where $B_r=\sum_{n=0}^r S(r,n)$ is the $r$-th Bell number. In terms of partitions, the prior used by \cite{sadinle_2014} can be written as 
$p(z(\lam)=z)=\frac{1}{B_r}$ and
$p(n(z)=n)=\frac{S(r,n)} {B_r}$.
This prior is also a special case of a KP model.

Moreover, the discrete uniform priors of \cite{steorts14smered, steorts??bayesian, steorts15entity} can also be represented easily as KP models. We refer to \cite{steorts??bayesian} for details.

We now provide a general theorem, which gives a relationship regarding priors (or partitions or blocks) of KP models to our KL divergence bounds. Suppose the prior on the linkage structure can be represented as a KP model \cite{pitman}. Then a wide class of priors is able to be considered and compared. 
\begin{theorem}
\label{thm-gibbs}
Consider model \ref{model:cat} or \ref{model:string}. Then the error bounds behave like the corresponding bounds in Theorem \ref{theorem:cat} or \ref{theorem:string}.
\end{theorem}
\begin{proof}
The results directly follows from the proofs of Theorems \ref{theorem:cat}, \ref{theorem:string} and the representation of KP models \cite{pitman}.  Specifically, all bounds in Theorems \ref{theorem:cat} and \ref{theorem:string} depend upon the linkage structure $\lam$, which in the proofs, is agnostic to its form. 
 \end{proof}

\subsection{Microclustering and Record Linkage}
There has been early work in Bayesian nonparametrics to push forward record linkage. The work of \cite{miller15microclustering, zanella2016microclustering} recently pointed out that most clustering tasks assume the cluster sizes grow linearly with the number of the data points. Such examples include infinitely exchangeable clustering models, including finite mixture models, Dirichlet process (DP) mixture models, and Pitman--Yor process (PYP) mixture models. However, in record linkage, such an assumption is undesirable since linkage methods require models that yield clusters whose sizes grow sublinearly with the total number of data points (records). Due to this, \cite{zanella2016microclustering}  defined the microclustering property as well as a new model exhibiting such growth, where their models outperformed or performed as well as the PYP and DP in terms of standard record linkage evaluation metrics on data for official statistics, medical data, and human rights data. Furthermore, the authors proved that one of their models satisfies the microclustering property under very weak assumptions. We refer to their paper for further details.

One insight of this paper was the fact that their class of microclustering models considered can always be written as 
a KP model. Combining this clustering approach with the likelihood in equation \ref{model:cat} or \ref{model:string} immediately allows one to perform record linkage inference. Furthermore, Theorems \ref{theorem:cat} and \ref{theorem:string} are immediately satisfied since the prior on the linkage structure can be represented as a KP model.

\section{Simulation Study and Discussion}

\label{sec:experiments}

We consider how the bounds in Sections \ref{sec:cat} and \ref{sec:str} hold for two simulated experiments. In our experiments (\textbf{Experiment I} and \textbf{Experiment II}), synthetic categorical data are generated according to either model \ref{model:cat} or \ref{model:string} using the parameters shown in Table \ref{table:params} and \ref{table:params-str}, respectively. In order to 
consider a realistic set of strings for $S$, we consider the set of 20 most popular female baby names from 2014, according to the United States Census. Then for the distance $d$, we consider the generalized Levenshtein edit distance.

We then generate both categorical and string records according to either model \ref{model:cat} or \ref{model:string}. For each experiment, we vary exactly one of the parameters to demonstrate its impact of the linkage error rate $Pr( (\hat{\Lambda}_{ij}, \bY) \ne (\Lambda_{ij}, \bY))$. We choose the other values such that the performance is neither extremely low nor extremely high. We set the distortion parameter $\beta_\ell$ to the same value for each $\ell$, i.e.\ $\beta_\ell = 0.6$ denotes a distortion probability of 0.6 for every field. $\beta_\ell = $ 0.0 to 1.0 means we started with $\beta_\ell = 0$ for all $\ell$ and swept the values until $\beta_\ell = 1$ for all $\ell$. Recall $p$ is the number of fields, and thus the maximum value of $\ell$. We also set each $\theta_{\ell m}$ to the same value, i.e.\ $\theta_{\ell m} = 0.1$ denotes $\theta_{\ell m} = 0.1$ for all $\ell$ and all $m$. This further implies each field $\ell$ takes on exactly $M_\ell = 1/\theta_{\ell m}$ values in order for $\theta_\ell$ to be a valid probability distribution.

\begin{table}
\begin{center}
  \begin{tabular}{ l c c c c}
    Experiment & $N$ & $\beta_\ell $ & $p = p_c$  & $\theta_{\ell m}$ \\ \hline
    Fig. 1(a) & 10 to 500 & 0.6 & 3 & 0.1 \\
    Fig. 1(b) & 100  & 0 to 1 & 3 & 0.1 \\
    Fig. 1(c) & 100  & 0.6 & 1 to 8& 0.25 \\
    Fig. 1(d) & 100  & 0.8  & 5& $\frac{1}{46} \text{ to } 1$
  \end{tabular}
  \caption{Categorical Experiments}
  \label{table:params}
  \end{center}
\end{table}

\begin{table}
\begin{center}
  \begin{tabular}{ l c c c c c}
    Experiment & $N$ & $\beta_\ell $ & $p = p_s$ & $c$ \\ \hline
    Fig. 2(a) & 100 to 500 & 0.6  & 1 &  1.0 \\
    Fig. 2(b) & 100  & 0.2 to 1 & 1  & 1.0 \\
    Fig. 2(c) & 100  & 0.6 & 1 to 10  & 1.0 \\
    Fig. 2(d) & 100  & 0.6 & 1 & 0 to 2
  \end{tabular}
  \caption{String Experiments}
  \label{table:params-str}
  \end{center}
\end{table}

We compare the bound in Theorem \ref{theorem:cat} to two record linkage algorithms \cite{steorts14smered, steorts??bayesian, steorts15entity}. The first is an exact sampler, which samples directly from $Pr(\Lambda_{ij} | X_{ij}, \bY, \bz)$. The second is a more realistic Gibbs sampler with empirically motivated priors proposed by \cite{steorts15entity}. We run the Gibbs sampler for 10,000 iterations on all experiments to ensure proper mixing. There is some difficulty in comparing $\Lambda$ to $\hat \Lambda$, as there are multiple equally correct modes due to arbitrary re-orderings of the latent individuals $\hat \bY$ and corresponding linkage structure $\hat \Lambda$. Even though the Gibbs sampler may infer the correct latent individuals $\bY$ and linkage structure, because the ordering is arbitrary, it is unlikely that $\Lambda = \hat \Lambda$. To avoid such an issue of label switching, we fix $\hat \bY$ during the sampling process. 
%We have tried to hold these fixed during the burn-in period, and then release them, but the Gibbs sampler escapes the mode of the correct ordering.

Specifically, we compare the bound to the empirical error rate of the Gibbs sampler proposed by \cite{steorts15entity}. In order to compute the empirical probability $Pr( \hat{\Lambda}_{ij} \ne \Lambda_{ij})$, we hold $\bY$ fixed during Gibbs sampling to ensure errors in $\hat \lam$ are not due to arbitrary changes in the ordering of the labels of $\bY$. In addition, we compare the linkage error rate to an exact sampler, which samples directly from $Pr(\lam | X, \bY, \bz)$.  %In the former situation, for string values of records, we consider the set of most popular newborn names from the US Census. We compare the bound in Theorem \ref{theorem:string} to the linkage error rate of the Gibbs sampler and exact sampler. Results are shown in Figures \ref{fig:N_string}-\ref{fig:theta_values_string}. 

\paragraph{Results of Experiment I}
In Figures \ref{fig:experiments} (a)-(d) we vary the number of records $N$, distortion parameter $\beta$, number of fields $p$ and number of values each field takes $M_\ell$, respectively. The empirical results demonstrate Theorem \ref{theorem:cat} captures the dependence between the error rate and the all relevant latent parameters $\theta$, $N$ and $\beta$. Specifically, linking records becomes more difficult as $N$ increases, the distortion parameter $\beta$ increases, the number of fields $p$ decreases or the number of values each field can take $M_\ell$ decreases. The bound nicely captures the logarithmic increase in error with respect to $N$ in Figure \ref{fig:experiments} (a), which gives hope for linking records in very large databases. Other terms appear to be $\bar O(n)$ when not near extreme error values, implying low noise and a larger feature space are essential to performing high quality record linkage.

\begin{figure}
  \begin{minipage}[t]{1\textwidth}
    \hspace{-4mm}
    \includegraphics[width=0.28\textwidth]{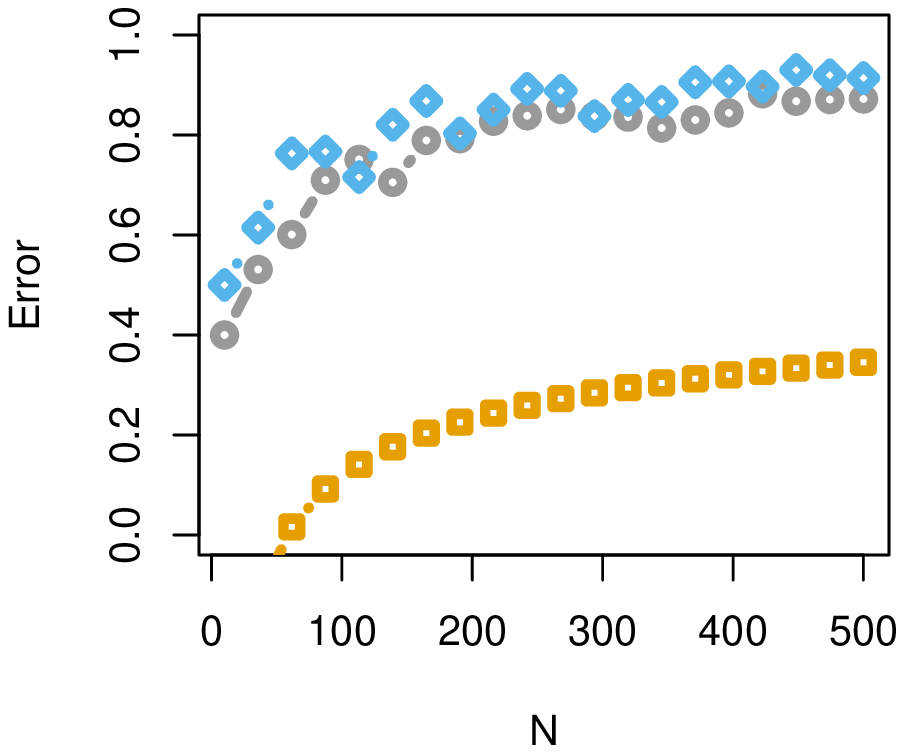}
    \label{fig:N}
    \hspace{-6mm}
    \includegraphics[width=0.28\textwidth]{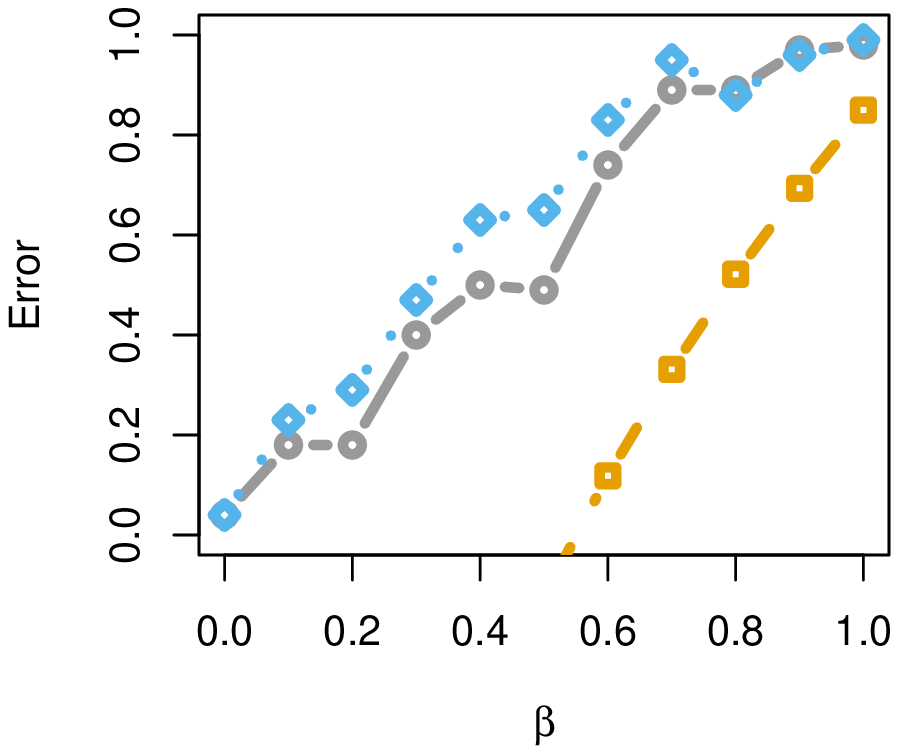}
    %\label{fig:beta}
  \end{minipage}
  
   \vspace{-10mm}
  \begin{minipage}[t]{1\textwidth}
    \hspace{-4mm}
 \includegraphics[width=0.28\textwidth]{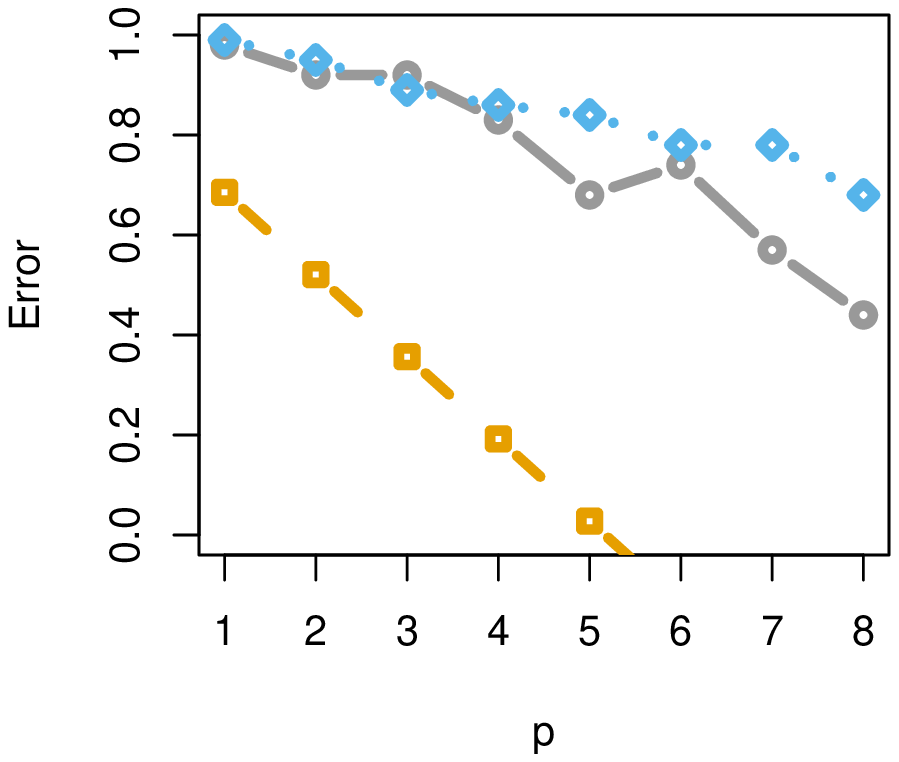}
     \hspace{-6mm}
    \includegraphics[width=0.28\textwidth]{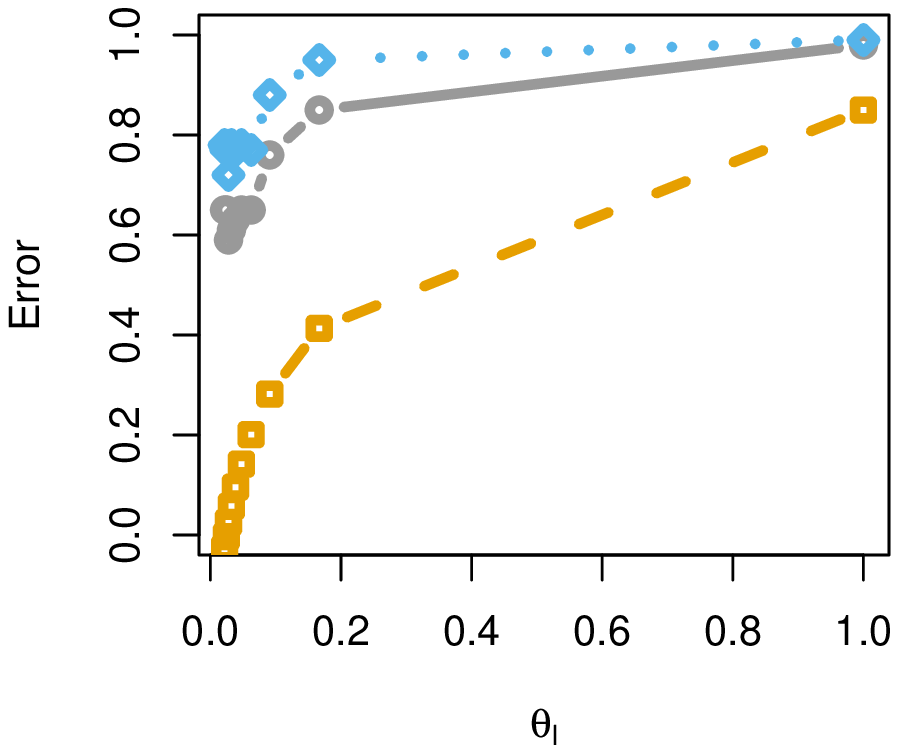}
    %\label{fig:theta_values}
  \end{minipage}
  
    \caption[]{Theorem \ref{theorem:string} (gold squares) holds on simulated
    categorical records
    % (top row) and string records (bottom row) 
    compared to exact sampling (grey circles) and Gibbs sampler (blue diamonds).}
       \label{fig:experiments}
\end{figure}

\paragraph{Results of Experiment II}
Figures \ref{fig:experiments-eb} (a)-(d) show Theorem \ref{theorem:string} is tight to the true performances on string data when varying $N$, $\beta$, number of string fields $p_s$ and $c$, respectively. As expected, and similarly to the categorical results, linking records becomes more difficult as $N$ increases, the distortion parameter $\beta$ increases and the parameter $c$ decreases. The effects of parameter variation is less noticeable in the string experiments due to the fact that linking string fields is easier than ones that have been anonymized, i.e., categorical fields. 

\begin{figure}
  \begin{minipage}[t]{1\textwidth}
    \hspace{-4mm}
    \includegraphics[width=0.28\textwidth]{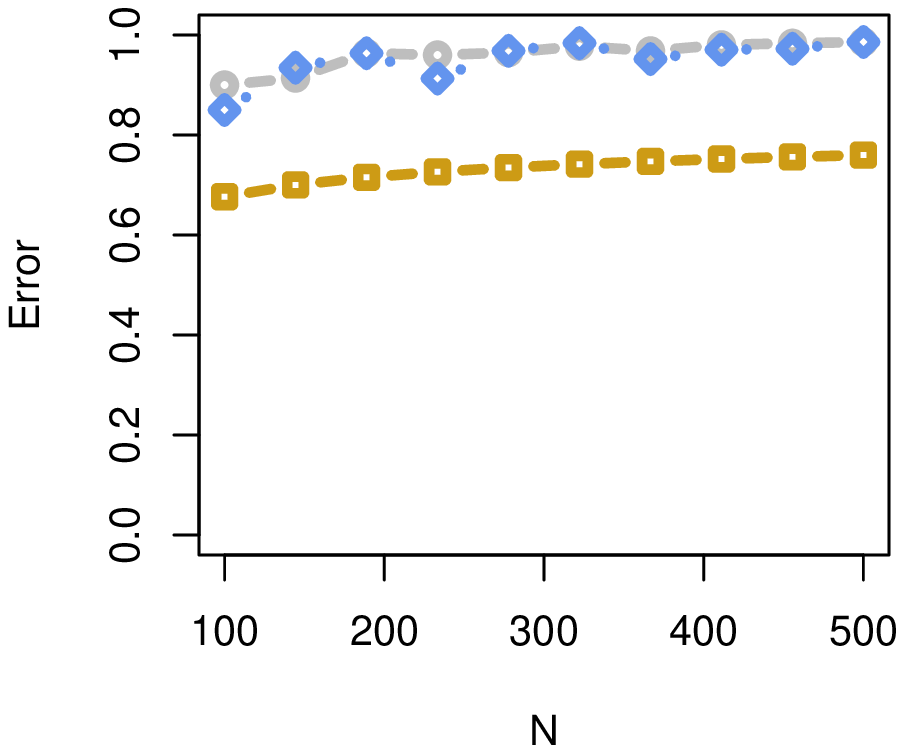}
    \label{fig:N}
    \hspace{-6mm}
    \includegraphics[width=0.28\textwidth]{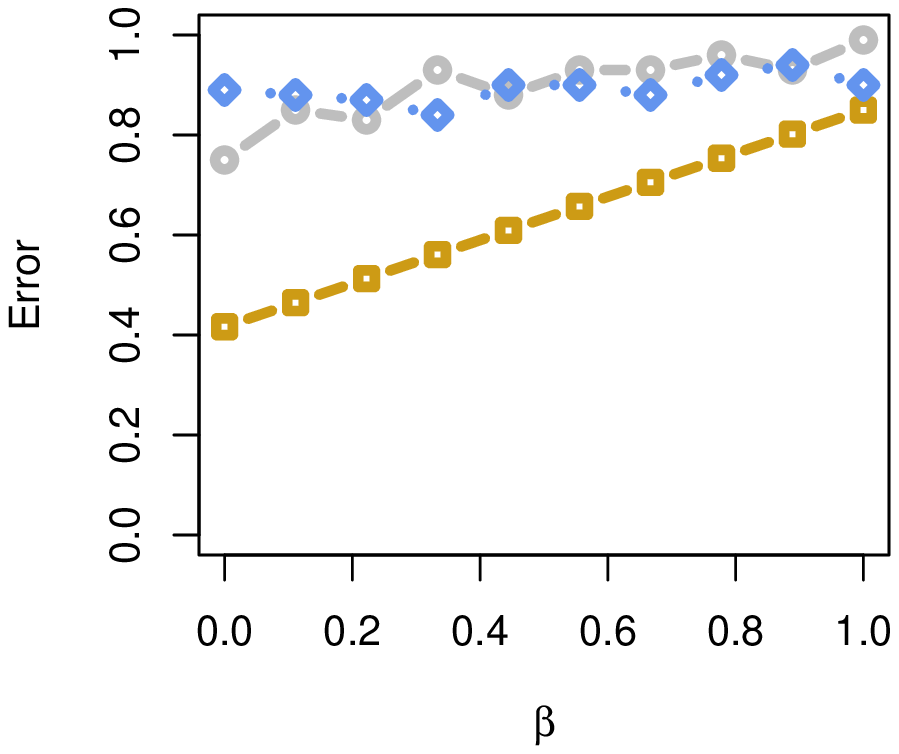}
    %\label{fig:beta}
  \end{minipage}
  
   \vspace{-10mm}
  \begin{minipage}[t]{1\textwidth}
    \hspace{-4mm}
 \includegraphics[width=0.28\textwidth]{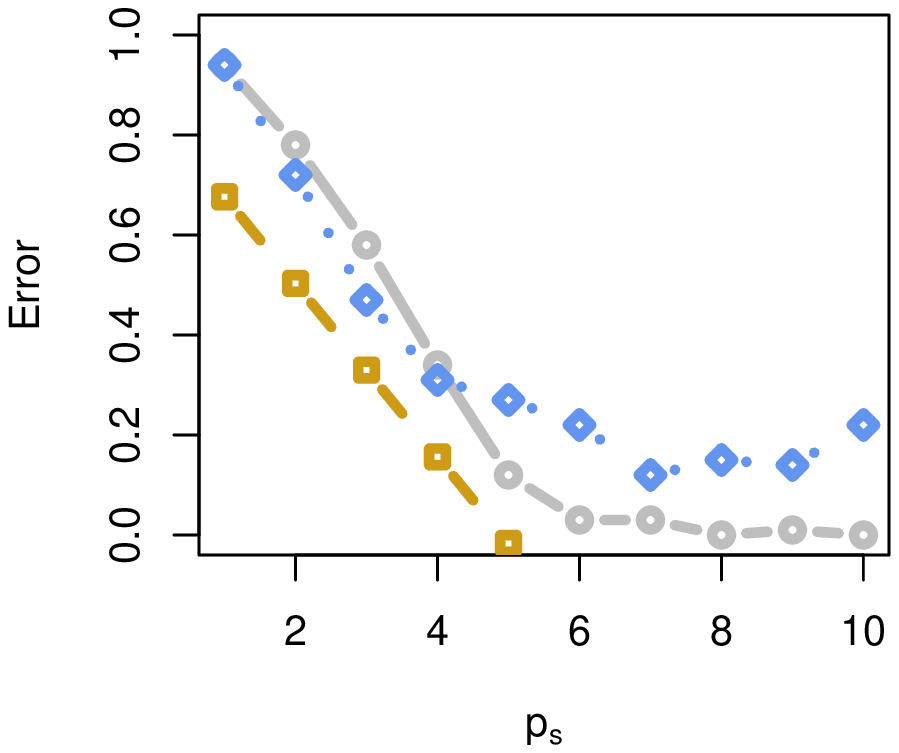}
     \hspace{-6mm}
    \includegraphics[width=0.28\textwidth]{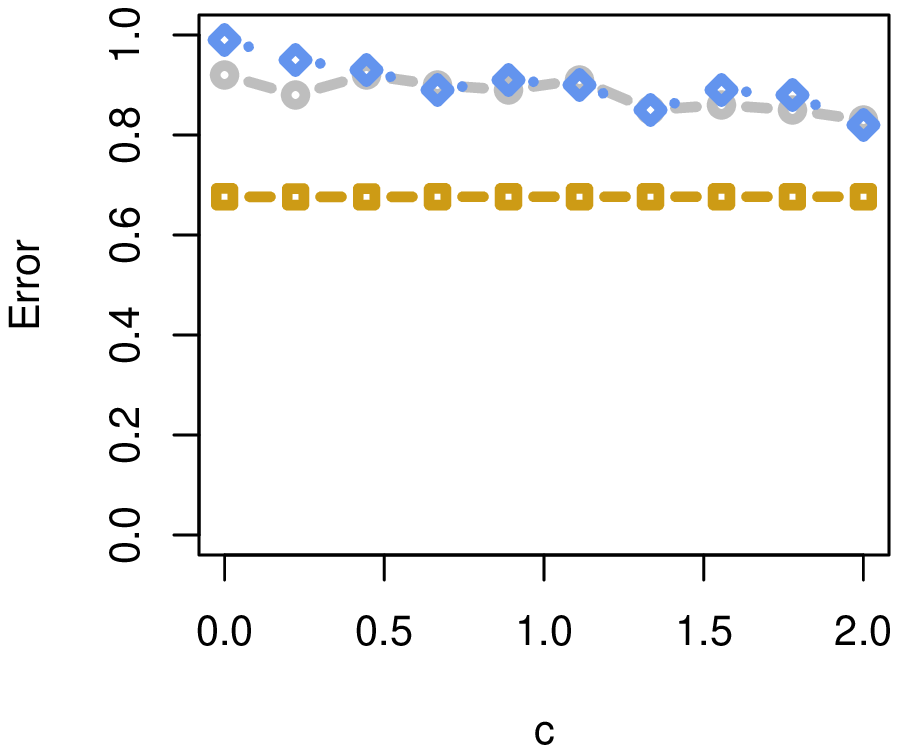}
    %\label{fig:theta_values}
  \end{minipage}
  
    \caption[]{Theorem \ref{theorem:string} (gold squares) holds on simulated
    noisy string records
    % (top row) and string records (bottom row) 
    compared to exact sampling (grey circles) and Gibbs sampler (blue diamonds).}
       \label{fig:experiments-eb}
\end{figure}

%\textcolor{red}{[[To do: We need to decide as a group if there are any additional experiments we want to run. I suggest at least one more string experiment so we have 4 total. We also need to decide what to do about Gibbs sampler crashing on $c$ experiments.]]}

The Gibbs sampler (blue diamonds) performs almost as well as the exact sampler (grey circles). In fact, due to the conditional entropy version of Fano's inequality and the fact that $H(X | Y) \leq H(X)$, any Gibbs sampler cannot perform better in expectation than an exact sampler. Thus, we believe the gap between the bound (gold squares) and the exact sampler does not necessarily indicate the existence of a better algorithm, but perhaps only some unnecessary slack due to the application of Pinsker's and then reverse Pinsker's inequalities.

\subsection{Discussion of Results}
\label{sec:discussions}
As illustrated in Theorems \ref{theorem:cat} and \ref{theorem:string}
%under the methods of \cite*{steorts_2014_aistats,steorts??bayesian,steorts_2014_eb}, 
we have derived an upper bound on the KL divergence as well as lower bounds for misclassifying a latent entity.
In Theorem~\ref{theorem:cat} (i), we showed that the latent entities become more distinct when $\gamma$ is increasing. This is in contrast to when $\gamma$ gets closer to 0, since then the latent entities become more similar. In Theorem~\ref{theorem:cat} (ii), we showed that as the distortion parameter $\beta_\ell \rightarrow 1,$ then the upper bound $\gamma$ is infinite. In practice, as illustrated in \cite{steorts14smered}, the latent entities are difficult to distinguish when the amount of distortion is more than 5\% at every field value.  Thus, this corresponds to when the bound is too loose.  On the other hand,  as $\beta_\ell \rightarrow 0,$ the latent entities become more separated. 

We discuss how separated the latent entities are under choices of $\beta_\ell, \theta_\ell$ and $N$, providing guidance to the user in this setting given our simulation results. As practical guidance when the distortion is between 0 to $5\%$ at every feature value, the latents will be more separated and the bound will be be loose. On the other hand, as $\beta_\ell$ increases, the bound becomes tighter. The choice of $\beta_\ell$ can be made using subjective information about the underlying data and tuned using the hyper-parameters $a,b$. (See \cite{steorts14smered, steorts15entity} for choosing such values). On the other hand, we can see that for more realistic values of the distortion parameter in Figure \ref{fig:experiments} (a), (b), and (d) , the bound is quite loose when the distortion parameter $\beta_\ell$ is large. Thus, a loose bound here is warranted due to the amount of noise or model-misspecification being placed into the model as well as the fact that all of the fields being used are categorical. Such results match the intuition given in \cite{steorts14smered}.

%This can be tuned using $a,b.$ A
%
%
%s can be seen from Figure 1(a), (b) and (d) the bound is quite loose when the distortion parameter $\beta_\ell$ is large. Thus, a loose bound here is warranted due to the amount of noise or model-misspecification being placed into the model. On the other hand, we can see that for more realistic values of the distortion parameter in Figure 1 (b), the bound is tighter. Such results match the intuition given in \cite{steorts14smered}.

%\textcolor{red}{Can I say much about Theorem~\ref{theorem:string} (i) and any asymptotics?}. 

In Theorem~\ref{theorem:string} (ii), we derived a lower bound where the minimum probability of getting a latent entity wrong is controlled by $c,$ which is determined by the moment generating function of the distances between an observed string and a latent string. This bound has the same type of form as the bound in Theorem~\ref{theorem:cat}, however, since we now have string-valued data, we see that the minimum probability of getting a latent entity wrong is dominated by the string-valued variables and specifically, the distances functions used and the constants used. In comparison to \cite{steorts15entity}, this completely matches up with the sensitivity that was seen to the choice of the distance functions as well as the choice of $c$ as this will completely dominate the posterior, and hence, the ability to tell latent entities apart under this posterior. 

{In practice, the driving force of the tightness of the bound is $c$, the steepness parameter of the string distribution in equation 2. As $c$ increases, it is less likely for a string-valued record's features to be distorted to values that are far from that of their latent feature values. This is verified in Figure \ref{fig:experiments-eb}(c), where linkage error decreases as $c$ increases. The work of \cite{steorts15entity} gave practical choices for $c$, which were [0,2].
%\textcolor{red}{Matt: Bound is actually tighter at higher $c$, see Figure 2c RCS: the values of c seems to be for all practical purposes. 5 is too large}. 
Similarly, we can speak to the tightness of $d$, which relies on the distortion parameter $\beta_\ell$ not being too small in practice, as verified in Figure \ref{fig:experiments-eb}(b). In terms of the bounds found in Theorem 1 and 2, the empirical Gibbs sampler has tight bounds in almost all situations, except when the number of features is large, $N$ is too small, or $\beta_\ell$ is too small (and similarly for $\theta_\ell$). This coincides with exactly what we would expect in practice from the real experiments of \cite{steorts15entity}. 

For all applications in both categorical and string data, we expect the bounds to be as loose in practice (corresponding to easier record linkage), when the distortion parameter is small ($0-4$) and when the the number of fields is large ($p \geq 5$) or the number of values that each field can take, $M_\ell$, increases (this will be application specific). Finally, the bounds should be tighter, corresponding to more difficult record linkage, as the total number of records N increases (see Figure \ref{fig:experiments}). These parameter values match almost exactly with two real data experiments (corresponding ranges of  parameters) as well as a simulation study from \cite{steorts14smered, steorts??bayesian}.

%In addition, while we consider a particular newer class of Bayesian entity resolution models, we are able to make relationships to other works in the literature through the linkage structure, which we do in section~\ref{sec:priors-on-partitions}. 
%%  These connections weave together methods, applications, and theory between a series of very recent papers, illustrating how our bounds  practical!
}

\section{Future Directions}
\label{sec:discussion}
%\textcolor{red}{redo this.}
First, we have derived general
performance bounds for record linkage, making connections to KP models and other related Bayesian models. More specifically, we have drawn connections to a wide class of models from Bayesian record linkage. Second, our bound for the categorical Bayesian record linkage model is easily interpretable and matches the intuition of the generative model. Third, our bound for the categorical and noisy string model, takes a similar form to that of the categorical model. We are also able to interpret  this bound in a way that aligns with the interpretations \cite{steorts15entity, steorts14smered, steorts??bayesian} as well as show the practicality of our bounds to the aforementioned papers. More specifically, our bounds are empirically loose for categorical data, which is not unexpected since there is little information available to match on. This contrasts the empirical tight bounds for both categorical and noisy string data. As illustrated in our experiments, with just one string variable, our bounds become much tighter, and as the number of strings increases, the bound becomes more tighter when compared to exact and Gibbs sampling. 

%In terms of future work in this area, 
In addition, there has been early work in Bayesian nonparametrics to push forward record linkage. The work of \cite{miller15microclustering} pointed out that most clustering tasks assume cluster sizes grow linearly with the size of the database. Such examples include infinitely exchangeable clustering models, including finite mixture models, Dirichlet process mixture models, and Pitman--Yor process mixture models, which all  make this linear growth assumption. However, in record linkage such an assumption is  undesirable since linkage methods require models that yield clusters whose sizes grow sublinearly with the total number of data points. This observation led the authors to define the microclustering property as well as a new model exhibiting such growth. Our work has been able to provide bounds for the aforementioned work since the prior consider is a KP model. 
In future work it would be helpful to try and draw connections between those proposed in  \cite{miller15microclustering} and \cite{steorts15entity, steorts14smered, steorts??bayesian} in order to generalize such bounds and provide tighter bounds using conditional entropy or other sophisticated bounding methods. 

\subsubsection*{Acknowledgments}
We thank David Choi, David Dunson, David Banks, and the reviewers for improving the ideas that led to publication of this paper.  This work was supported in
part by NSF grants SES-1534412 and SES-1131897, DARPA FA8750-12-2-0324 and FA8750-14-2-0244.
\looseness=-1

% It seems reasonable that connections can be drawn between these two papers in a general way. Furthermore, given a generalization, there seems much potential then for deriving general and practical performance bounds for both hierarchical and nonparametric Bayesian methods. Such generalizations in terms of methods and in useful bounds will only help a growing field. 

\clearpage
\newpage

\bibliographystyle{plain}
\bibliography{references}

\begin{thebibliography}{10}

\bibitem{banerjee2005}
Arindam Banerjee, Srujana Merugu, Inderjit~S Dhillon, and Joydeep Ghosh.
\newblock Clustering with bregman divergences.
\newblock {\em The Journal of Machine Learning Research}, 6:1705--1749, 2005.

\bibitem{berend_2014}
Daniel Berend, Peter Harremo{\"e}s, and Aryeh Kontorovich.
\newblock Minimum {KL}-divergence on complements of $l\_1$ balls.
\newblock {\em IEEE Transactions on Information Theory}, 60(6):3172--3177,
  2014.

\bibitem{christen_2011}
P.~Christen.
\newblock {\em Data Matching: Concepts and Techniques for Record Linkage,
  Entity Resolution, and Duplicate Detection}.
\newblock Springer, 2012.

\bibitem{copas_1990}
J.~Copas and {F.J.} Hilton.
\newblock Record linkage: Statistical models for matching computer records.
\newblock {\em Journal of the Royal Statistical Society, Series A},
  153(3):287--320, 1990.

\bibitem{donoho2006}
David~L Donoho.
\newblock Compressed sensing.
\newblock {\em Information Theory, IEEE Transactions on}, 52(4):1289--1306,
  2006.

\bibitem{gutman_2013}
R.~Gutman, C.~Afendulis, and A.~Zaslavsky.
\newblock A {B}ayesian procedure for file linking to analyze end- of-life
  medical costs.
\newblock {\em Journal of the American Statistical Association},
  108(501):34--47, 2013.

\bibitem{kullback_1951}
Solomon Kullback and Richard~A Leibler.
\newblock On information and sufficiency.
\newblock {\em The Annals of Mathematical Statistics}, 22(1):79--86, 1951.

\bibitem{liseo_2013}
B.~Liseo and A.~Tancredi.
\newblock Some advances on {B}ayesian record linkage and inference for linked
  data.
\newblock {\em Technical Report}, 2013.

\bibitem{miller15microclustering}
Jeffrey Miller, Brenda Betancourt, Abbas Zaidi, Hanna Wallach, and Rebecca
  Steorts.
\newblock {The Microclustering Problem: When the Cluster Sizes Don't Grow with
  the Number of Data Points}.
\newblock {\em NIPS Bayesian Nonparametrics: The Next Generation Workshop
  Series}, 2015.

\bibitem{pitman}
Jim Pitman.
\newblock {\em Combinatorial Stochastic Processes: Ecole D'Et{\'e} de
  Probabilit{\'e}s de Saint-Flour XXXII-2002}.
\newblock Springer, 2006.

\bibitem{prelov2008}
Vyacheslav~Valer'evich Prelov and Edward~C. van~der Meulen.
\newblock Mutual information, variation, and fano's inequality.
\newblock {\em Problems of Information Transmission}, 44(3):185--197, 2008.

\bibitem{sadinle_2014}
Mauricio Sadinle.
\newblock Detecting duplicates in a homicide registry using a {B}ayesian
  partitioning approach.
\newblock {\em The Annals of Applied Statistics}, 8(4):2404--2434, 2014.

\bibitem{shannon1948}
Claude~E Shannon.
\newblock A note on the concept of entropy.
\newblock {\em Bell System Tech. J}, 27:379--423, 1948.

\bibitem{steorts15entity}
R.~C. Steorts.
\newblock Entity resolution with empirically motivated priors.
\newblock {\em Bayesian Analysis}, 10(4):849--875, 2015.

\bibitem{steorts14smered}
R.~C. Steorts, R.~Hall, and S.~E. Fienberg.
\newblock {SMERED}: A {B}ayesian approach to graphical record linkage and
  de-duplication.
\newblock {\em Journal of Machine Learning Research}, 33:922--930, 2014.

\bibitem{steorts??bayesian}
R.~C. Steorts, R.~Hall, and S.~E. Fienberg.
\newblock A {B}ayesian approach to graphical record linkage and de-duplication.
\newblock {\em Journal of the American Statistical Society}, In press.

\bibitem{liseo_2011}
A.~Tancredi and B.~Liseo.
\newblock A hierarchical {B}ayesian approach to record linkage and population
  size problems.
\newblock {\em Annals of Applied Statistics}, 5(2B):1553--1585, 2011.

\bibitem{winkler_2006}
W.~E. Winkler.
\newblock Overview of record linkage and current research directions.
\newblock Technical report, U.S. Bureau of the Census Statistical Research
  Division, 2006.

\bibitem{zanella2016microclustering}
Giacomo Zanella, Brenda Betancourt, Jeffrey~W Miller, Hanna Wallach, Abbas
  Zaidi, and Rebecca Steorts.
\newblock Flexible models for microclustering with application to entity
  resolution.
\newblock In {\em Advances in Neural Information Processing Systems}, pages
  1417--1425, 2016.

\end{thebibliography}

\clearpage
\newpage

\appendix

%\section*{Supplementary Material}
%\label{sec:appendix}

\section{Example of the Record Linkage Process} 
\label{sec:toy}
We provide a toy illustration of the general record linkage process in figure~\ref{fig:linkageProcess}. Consider three databases $D_1, D_2, D_3$ and the notation already introduced, {where here $k=3$}. Suppose the ``population'' entities have four members, where name and address are stripped for anonymity and they are listed by state, age, and sex, as is often the case with de-identified data. 

For instance, assume the true latent entity vector $\bm{y}$ is \emph{known}:
            \[
         \textcolor{black}{\bm{y}}=
            \left[ {\begin{array}{c}
              \text{NC, 72, F} \\
    \text{SC, 73, F} \\
    \text{PA, 91, M}\\
    \text{VA, 94, M}
                \end{array} } \right]
        .\]

The observed records $\bm{X}$ are given in three separate databases (k=3), which would combine into a three-dimensional array.  We write this here as three two-dimensional arrays for notational simplicity:
\[
D_1=\begin{bmatrix}\text{NC, 72, F}\\ \text{SC, \textcolor{red}{70}, F}\\\text{PA, 91, M}\end{bmatrix},
D_2=\begin{bmatrix}\text{SC, \textcolor{red}{37} , F}\\
\text{VA, \textcolor{red}{93}, M}\\
\text{PA, \textcolor{red}{92}, M}\end{bmatrix},
\]
$$
D_3 =
\begin{bmatrix}\text{NC, 72 , F}\\
\text{NC, 72, F}\\
\text{SC, \textcolor{red}{72}, F}\\
\text{VA, 94, M}
\end{bmatrix}.
$$
Here, for the sake of keeping the illustration simple, only age is distorted.
Comparing $\bm{X}$ to $\bm{y}$, the intended linkage and distortions are 
$$
\bm{\Lambda}=\begin{bmatrix} 1 & 2 & 3 &  \\
2 & 4 & 3 &  \\
1 & 1 & 2 & 4 \\
\end{bmatrix},\qquad
$$
$$
\bm{z_1}=  \begin{bmatrix}
0 & 0 & 0\\
0 & 1 & 0 \\
0 & 0 & 0 \\
\end{bmatrix}, 
\bm{z_2}=  \begin{bmatrix}
0 & 1 & 0\\
0 & 1 & 0 \\
0 & 1 & 0\\\end{bmatrix}, 
\bm{z_3}=  \begin{bmatrix}
0 & 0 & 0\\
0 & 0 & 0 \\
0 & 1 & 0\\
0 & 0 & 0 \\
\end{bmatrix}.
$$

In this linkage structure, every entry of $\bm{\Lambda}$ with a value of~2 means that some record from~$\bm{X}$ refers to the latent entity with attributes ``SC, 73, F."  Here, the age of this entity is distorted in all three databases, as can be seen from $\bm{z}$.  (Note that $\bm{z}$, like $\bm{X}$, is also really a three-dimensional array.)  Looking at $\bm{z}_1$ and $\bm{z}_3$, we see that there is only a single record in either list that is distorted, and it is only distorted in one field.  In list 2, however, every record is distorted, though only in one field.

Figure \ref{fig:linkageProcess} illustrates the interpretation of the linkage structure as a bipartite graph in which each edge links a record to a latent entity.
For clarity, figure \ref{fig:linkageProcess} shows that $X_{11}$ and
$X_{22}$ are the same entity and shows that $X_{13}, X_{21},$ and $X_{34}$
correspond to the same entity. The rest are non-matches (or singleton entities).

\begin{figure}[htbp]
\begin{center}
\includegraphics[width=0.35\textwidth]{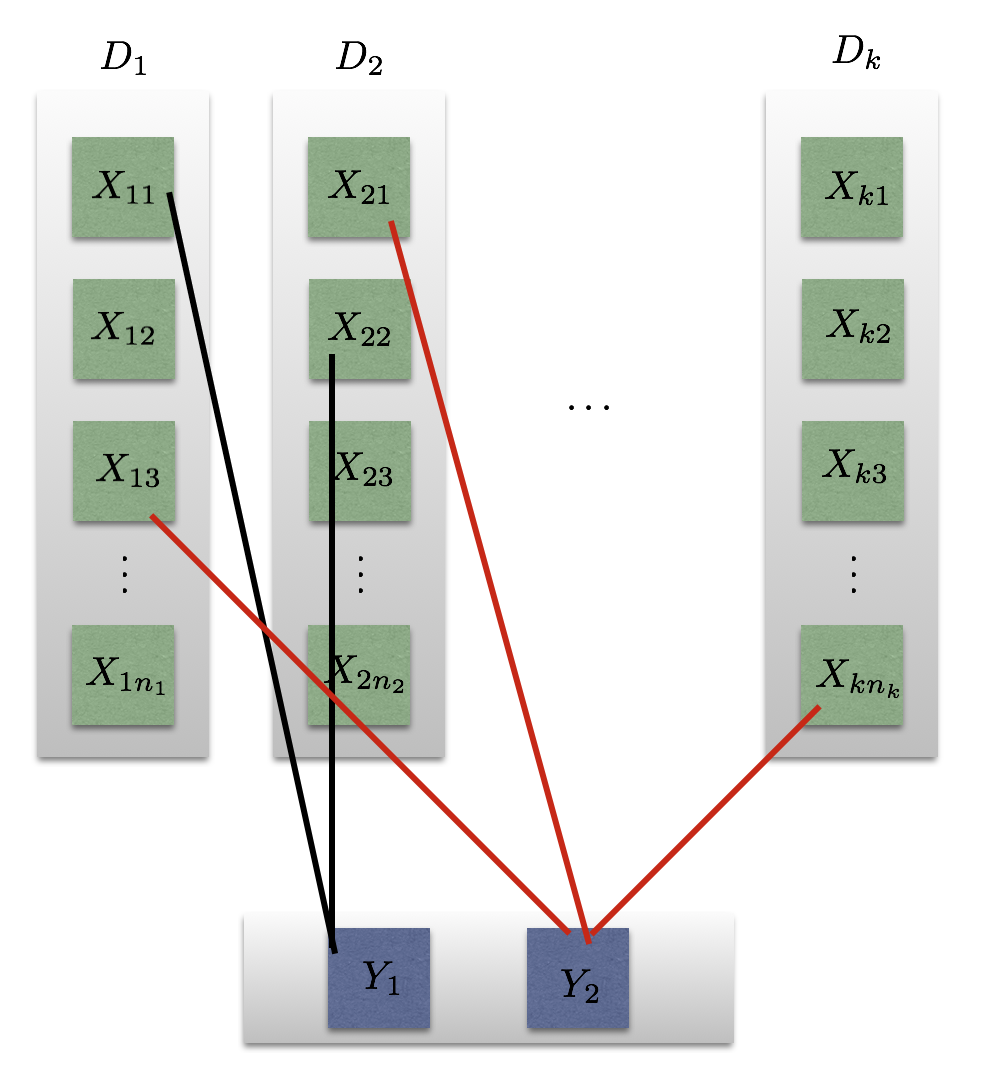}
\caption{A general illustration of the record linkage process. We assume databases $D_1,\ldots D_k$.
We assume records $\bm{X}$ that we cluster to latent entities $\bm{Y}$. Records that belong to the same same latent entity are kept track of using the linkage structure
 $\lam.$}
\label{fig:linkageProcess}
\end{center}
\end{figure}

%\section{Graphical Representation of Record Linkage Process} 
%We provide a graphical representation of the record linkage models considered in equations \ref{model:cat} and \ref{model:string} in Figure~\ref{fig:graphicalProcess}. 
%
%
%
%\begin{figure}[htbp]
%\begin{center}
%\includegraphics[width=0.35\textwidth]{figures/recordLinkage_graphicalModel}
%\caption{Add this. }
%\label{fig:graphicalProcess}
%\end{center}
%\end{figure}

\section{Derivation of Theorem 1}
\label{app:theorem:string}
We briefly restate the theorem, and then provide its derivation. 

\begin{theorem}
\label{theorem:string}
Assume data $\bX,$ and distributions $P,Q \in \mathcal{P}.$  Assume two distinct linkage structures, denoted by $Y_{\Lambda_{ij}\ell}, Y_{\Lambda^\prime_{ij}\ell}.$
\begin{enumerate}
\item [i)] There is an upper bound on the KL divergence between any $P,Q \in \mathcal{P}$ 
given by $\kappa,$ that is $D_X(P||Q) \leq \kappa.$
%\begin{align}
%&D_{\bX}(P || Q)\\
% &\geq
%\sum_{i,j,\ell} 2(1 - \beta_\ell) \\
%& + \sum_{i,j,\ell}
%I(Y_{\Lambda_{ij}\ell} \neq Y_{\Lambda^\prime_{ij}\ell})
%\left(
%1 - e^{-c d(Y_{\Lambda_{ij}\ell}, Y_{\Lambda^\prime_{ij}\ell})}
%\right) E[ e^{-c  d(m, Y_{\Lambda_{ij}\ell})} ],
%\end{align}
\item [ii)] $Pr(\Lambda_{ij} \neq \Lambda_{ij}) \geq 1- \dfrac{\kappa + \ln 2}{\ln r},$
where 
\begin{align*}
\kappa &= \max_{\Lambda_{ij} \neq \Lambda^\prime_{ij}}\bigg[
2 \sum_{\ell} (1-\beta_\ell) I(Y_{\Lambda_{ij}\ell} \neq Y_{\Lambda^\prime_{ij}\ell}) 
  +  \\
& \qquad \sum_{\ell m}  I(Y_{\Lambda_{ij}\ell} \neq Y_{\Lambda^\prime_{ij}\ell}) 
 \left(
1 - e^{-c d(Y_{\Lambda_{ij}\ell}, Y_{\Lambda^\prime_{ij}\ell})}
\right) \\
&\times E[ e^{-c  d(m, Y_{\Lambda_{ij}\ell})} ] \bigg]\ln\{ (\min Q)^{-1} \}
\end{align*}
and $r+1$ is the cardinality of $\mathcal{P}$.
% where the expectation is taken according to 
%random variable $M \sim \alpha_\ell.$ That is, $\sum_m  \alpha_\ell (m)
%e^{-c  d(m, m^\prime)} $ is the moment generating function of $d(M,m^\prime)$ (evaluated at c).
\end{enumerate}
\end{theorem}

\begin{proof}
We first prove (i). 
Consider
\begin{align}
%\label{eqn:string:mod}
&Pr(X_{ij\ell} =m \mid \bY, \lam, \bm{\theta}, \bm{\beta}) \notag \\
%&= 
%Pr(X_{ij\ell} =m, Z_{ij\ell} = 1 \mid \bY, \lam, \bm{\theta}, \bm{\beta}) \notag \\
%&\qquad + Pr(X_{ij\ell} =m, Z_{ij\ell} = 0 \mid \bY, \lam, \bm{\theta}, \bm{\beta}) \notag \\
&= Pr(X_{ij\ell} =m \mid \bY, \lam, \bm{\theta}, \bm{\beta}, Z_{ij\ell} = 1)  \notag \\
 &\times Pr( Z_{ij\ell} = 1 \mid  \bY, \lam, \bm{\theta}, \bm{\beta})\notag \\
&\qquad+ Pr(X_{ij\ell} =m \mid \bY, \lam, \bm{\theta}, \bm{\beta}, Z_{ij\ell} = 0) \notag \\
& \times Pr( Z_{ij\ell} = 0 \mid  \bY, \lam, \bm{\theta}, \bm{\beta})\notag \\
&\propto I(Y_{\Lambda_{ij}\ell} = m) (1-\beta_{\ell}) + \alpha_{\ell } (X_{ij\ell})\beta_{\ell} \notag \\
& \times \bigg[ 
\exp \{
-c\;d(X_{ij\ell}, Y_{\Lambda_{ij}\ell})
\}
\bigg]. \label{string:theFirstPart}
\end{align}
Suppose that $Y_{\Lambda_{ij}\ell} \neq Y_{\Lambda^\prime_{ij}\ell}$ 
%(as otherwise the divergence between $P$ and $Q$ is trivially~$0$).  
Equation \ref{string:theFirstPart} implies that
\begin{align}
D_{X_{ij\ell}}(P\| Q)
&\propto \sum_{m=1}^{M_\ell} 
I(Y_{\Lambda_{ij}\ell} = m) (1-\beta_{\ell})
 + \alpha_{\ell } (m)\beta_{\ell}  \notag \\
 & \times
\bigg[ 
e^{
-c\;d(X_{ij\ell}, Y_{\Lambda_{ij}\ell}) 
} \times \phi
\bigg],
\label{eqn:theSecondPart}
\end{align}
where $\phi = $
%$$\phi =\log\left[
$$\log\left[
\frac{
I(Y_{\Lambda_{ij}\ell} = m) (1-\beta_{\ell}) 
+ \alpha_{\ell } (m)\beta_{\ell} \bigg[ 
e^{
-c\;d(m, Y_{\Lambda_{ij}\ell})
}
\bigg]
}
{
I(Y_{\Lambda^\prime_{ij}\ell} = m) (1-\beta_{\ell}) 
+ \alpha_{\ell } (m)\beta_{\ell} \bigg[ 
e^{
-c\;d(m, Y^{\prime}_{\Lambda^\prime_{ij}\ell})
}
\bigg]}
\right]. 
$$

We now consider $\| P-Q\|_1$ and by equation \ref{eqn:theSecondPart}, we find
\begin{align}
\| P-Q\|_1
 &= \sum_{m\in M_\ell}\left| I(Y_{\Lambda_{ij}\ell} = m) (1-\beta_{\ell}) \right. \notag \\
&\quad \left. + \alpha_\ell(m) \beta_{\ell}
 \exp \{
-c\;d(m, Y_{\Lambda_{ij}\ell}^{})
\}
\right.\notag \\
&\left.\qquad-
I(Y_{\Lambda^\prime_{ij}\ell} = m) (1-\beta_{\ell}) \right. \notag \\
& \quad \left.- \alpha_\ell(m)  \beta_{\ell}
 \exp \{
-c\;d(m, Y_{\Lambda^\prime_{ij}\ell})
\}\right|. \label{eqn:theThirdPart}
\end{align}
Then by equation \ref{eqn:theThirdPart}, it is clear that
\begin{align*}
\| P-Q\|_1 &\leq 
 \sum_m  (1-\beta_{\ell})\left | \left[I(Y_{\Lambda_{ij}\ell} = m) -I(Y_{\Lambda^\prime_{ij}\ell} = m) \right] \right|\\
 &+ \sum_m  \alpha_\ell(m)  \beta_{\ell} \notag \\
& \times 
 \left| 
 \exp\{
 -c\; d(m, Y_{\Lambda_{ij}\ell})
 -
  \exp\{
 -c\; d(m, Y_{\Lambda^\prime_{ij}\ell})
 \}\right| \\
 & \leq 2 (1- \beta_\ell) + \beta_\ell \sum_m \alpha_\ell(m)   \notag \\
  & \times \left| 
 \exp\{
 -c\; d(m, Y_{\Lambda_{ij}\ell}) \right. 
 \left.  -
  \exp\{
 -c\; d(m, Y_{\Lambda^\prime_{ij}\ell})
 \} \right|.
\end{align*}
Now assume that two field attributes are different. That is, suppose there exists an $m \neq m^\prime.$ Then we assume that there exists a $\delta >0 $ such that $d(m,m^\prime) \geq \delta.$ 
%Then 
%we write 
%\begin{align}
%&\| P-Q\|_1 \\
% & \leq 2 (1- \beta_\ell) + \beta_\ell \sum_m \alpha_\ell(m) 
%  \left| 
% \exp\{
% -c\; d(m, Y_{\Lambda_{ij}\ell}) \right. \\
%& \left. \qquad -
%  \exp\{
% -c\; d(m, Y_{\Lambda^\prime_{ij}\ell})
% \} \right|.
%\end{align}
By the reverse triangle inequality, for any $m, m^\prime, m^{\prime \prime},$
%\begin{align}
%\label{eqn:reverse}
% | d(m, m^\prime) - d(m, m^ {\prime \prime}) | \notag
%&\leq d(m^\prime, m^ {\prime \prime})  \implies
%\end{align}
\begin{align}
\label{eqn:reverse}
| d(m, m^\prime) &- d(m, m^ {\prime \prime}) | 
\leq d(m^\prime, m^ {\prime \prime})  \implies \notag \\
& e^{-c [  d(m, m^\prime) -d(m, m^ {\prime \prime}) ] }
 \geq
e^{-c d(m^\prime, m^ {\prime \prime}) }.
\end{align}
Equation \ref{eqn:reverse}
 in turn implies that
\begin{align*}
&\sum_m \left[
\left(1 - e^{-c [  d(m, m^\prime) - d(m, m^ {\prime \prime}) ] }\right)
e^{-c d(m^\prime, m^ {\prime \prime}) } \alpha_\ell (m)
\right] \\
&\geq \sum_m
\left(1 - e^{-c [ d(m^\prime, m^ {\prime \prime}) ] }\right) 
e^{-c d(m^\prime, m^ {\prime \prime}) } \alpha_\ell (m).
\end{align*}

Then 
\begin{align*}
&\sum_m  \alpha_\ell (m)
\left[
e^{-c  d(m, m^\prime)} - e^{ -c d(m, m^ {\prime \prime}]}
\right] \\
 &= 
\sum_m  \alpha_\ell (m)
e^{-c  d(m, m^\prime)} 
\left(
1 - e^{-c d(m^\prime, m^ {\prime \prime})}
\right)  \\
&= 
\left(
1 - e^{-c d(m^\prime, m^ {\prime \prime})}
\right) 
\sum_m  \alpha_\ell (m)
e^{-c  d(m, m^\prime)}  \\
&= \left(
1 - e^{-c d(m^\prime, m^ {\prime \prime})}
\right)  E[ e^{-c  d(m, m^\prime)} ],
\end{align*}
%\begin{align*}
%\sum_m  \alpha_\ell (m)
%\left[
%e^{-c  d(m, m^\prime)} - e^{ -c d(m, m^ {\prime \prime}]}
%\right]
%& = 
%\sum_m  \alpha_\ell (m)
%e^{-c  d(m, m^\prime)} 
%\left(
%1 - e^{-c d(m^\prime, m^ {\prime \prime})}
%\right) \\
%&= 
%\left(
%1 - e^{-c d(m^\prime, m^ {\prime \prime})}
%\right) 
%\sum_m  \alpha_\ell (m)
%e^{-c  d(m, m^\prime)} \\
%&= \left(
%1 - e^{-c d(m^\prime, m^ {\prime \prime})}
%\right)  E[ e^{-c  d(m, m^\prime)} ]
%\end{align*}
where $M \sim \alpha_\ell.$ 

That is, 
$$\sum_m  \alpha_\ell (m)
e^{-c  d(m, m^\prime)}$$ 
is the moment generating function of $d(M,m^\prime)$ (evaluated at c), where $M \sim \alpha_\ell.$ 
%
%This implies that
%\begin{align*}
%&\sum_{i,j,\ell,m} \alpha_\ell (m) I(Y_{\Lambda_{ij}\ell} \neq Y_{\Lambda^\prime_{ij}\ell})
%\left[
%e^{-c  d(m, Y_{\Lambda_{ij}\ell})} - e^{ -c d(m, Y_{\Lambda^\prime_{ij}\ell}]}
%\right]\\
%&\quad\qquad\geq\sum_{i,j,\ell}
%I(Y_{\Lambda_{ij}\ell} \neq Y_{\Lambda^\prime_{ij}\ell})
%\left(
%1 - e^{-c d(Y_{\Lambda_{ij}\ell}, Y_{\Lambda^\prime_{ij}\ell})}
%\right) E[ e^{-c  d(m, Y_{\Lambda_{ij}\ell})} ].
%\end{align*}   
%We have established that for any   $Y_{\Lambda_{ij}\ell} \neq Y_{\Lambda^\prime_{ij}\ell},$ that the lower bound grows as $c$ goes to~$\infty,$ at a rate determined by the moment generating function of the distances.
This implies that 
\begin{align*}
&\| P-Q\|_1  \\
&\leq 2(1-\beta_\ell) 
+\beta_\ell \sum_m \\
& \qquad \times \left(
1 - e^{-c d(Y_{\Lambda_{ij}\ell}, Y_{\Lambda^\prime_{ij}\ell})}
\right) E[ e^{-c  d(m, Y_{\Lambda_{ij}\ell})} ].
\end{align*}
Then by reverse Pinker's inequality \cite{berend_2014}, we can write 
\begin{align*}
\max_{P, Q \in \mathcal{P}} &D_{\bX}(P||Q) \\
& \leq \max_{\Lambda_{ij} \neq \Lambda^\prime_{ij}}\bigg[
2 \sum_{ij\ell} (1-\beta_\ell) I(Y_{\Lambda_{ij}\ell} \neq Y_{\Lambda_{ij}\ell}^\prime) \\
 &+ \sum_{ij\ell m}  I(Y_{\Lambda_{ij}\ell} \neq Y_{\Lambda_{ij}\ell}^\prime) 
 \left(
1 - e^{-c d(Y_{\Lambda_{ij}\ell}, Y_{\Lambda^\prime_{ij}\ell})}
\right) \\
&  \times  \bigg[
E[ e^{-c  d(m, Y_{\Lambda_{ij}\ell})} ] \bigg] 
\times \ln\{ (\min Q)^{-1} \}\bigg] =: \kappa, 
\end{align*}
where
\begin{align*}
Q = I(Y_{\Lambda^\prime_{ij}\ell} = m) (1-\beta_{\ell}) - \alpha_\ell(m)  \beta_{\ell}
 \exp \{
-c\;d(m, Y_{\Lambda^\prime_{ij}\ell})
\}.
\end{align*}
Thus, (i) is established.
Using Fano's inequality, we find that
$$Pr(\hat \Lambda_{ij} \neq \Lambda_{ij}) \geq 1 -\frac{ \kappa + \ln 2}{\ln r}.$$

We have established that for any $Y_{\Lambda_{ij}\ell} \neq Y_{\Lambda^\prime_{ij}\ell}$, the minimum probability of getting a latent entity wrong is governed by the constant $c.$ That is, 
the lower bound grows as $c$ goes to~$\infty,$ and its rate of growth is determined by the moment generating function of the distances. We have now established (ii).
\end{proof}

\end{document}